\documentclass[11pt]{amsart}
\usepackage{amsmath}
\usepackage{amsthm}
\usepackage{hyperref,color}
\usepackage[margin=1.1in]{geometry}

\numberwithin{equation}{section}

\newtheorem{prop}{Proposition}
\newtheorem{lemma}[prop]{Lemma}

\newtheorem{thm}[prop]{Theorem}
\newtheorem{cor}[prop]{Corollary}

\numberwithin{prop}{section}

\theoremstyle{definition}
\newtheorem{defn}[prop]{Definition}

\newcommand{\del}{\partial}
\newcommand{\dt}{\frac{\partial}{\partial t}}
\newcommand{\brs}[1]{\left| #1 \right|}

\newcommand{\gG}{\Gamma}

\newcommand{\gd}{\delta}
\newcommand{\gs}{\sigma}

\newcommand{\gl}{\lambda}

\newcommand{\gw}{\omega}
\newcommand{\ga}{\alpha}
\newcommand{\gb}{\beta}

\newcommand{\N}{\nabla}
\newcommand{\nab}{\nabla}
\newcommand{\FF}{\mathcal F}

\renewcommand{\bar}[1]{\overline{#1}}

\newcommand{\lbr}{\left(}
\newcommand{\rbr}{\right)}

\newcommand{\hsp}{\hspace{0.5cm}}
\newcommand{\lap}{\Delta}
\newcommand{\la}{\lambda}

\newcommand{\ten}{\otimes}
\definecolor{grn}{rgb}{0,0.4,0}

\newcommand{\bRn}{\mathbb{R}^n}

\newcommand{\hook}{\mathbin{\hbox{\vrule height2.4pt width4.5pt depth-2pt
\vrule height5pt width0.4pt depth-2pt}}}

\DeclareMathOperator{\Rm}{Rm}

\DeclareMathOperator{\End}{End}

\begin{document}

\title[Entropy, Stability, and Yang-Mills flow]{Entropy, Stability, and
Yang-Mills flow}

\begin{abstract} Following \cite{CM}, we define a notion of entropy for
connections over $\mathbb R^n$ which has shrinking Yang-Mills solitons as
critical points.  As in \cite{CM}, this entropy is defined implicitly, making it
difficult to work with analytically.  We prove a theorem characterizing entropy
stability in terms of the spectrum of a certain linear operator associated to
the soliton.  This leads furthermore to a gap theorem for solitons.  These
results point to a
broader strategy of studying ``generic singularities" of Yang-Mills flow, and we
discuss the differences in this strategy in dimension $n=4$ versus $n \geq 5$. 
\end{abstract}

\date{\today}

\author{Casey Kelleher}
\email{\href{mailto:clkelleh@uci.edu}{clkelleh@uci.edu},
\href{mailto:jstreets@uci.edu}{jstreets@uci.edu}}
\author{Jeffrey Streets}

\address{Rowland Hall\\
         University of California\\
         Irvine, CA 92617}

\thanks{The first author was supported by an NSF Graduate Research Fellowship
DGE-1321846. 
The second author was partly supported by the National Science
Foundation DMS-1341836 and an Alfred P. Sloan Fellowship.}

\maketitle

\section{Introduction}

In \cite{CM} Colding and Minicozzi introduced a strategy for understanding the
mean curvature flow based on a notion of entropy-stability of singularities. 
Broadly speaking, the goal is to showing that all singularities other than
cylinders and spheres are unstable and hence can be perturbed away, leading to
the construction of ``generic" mean curvature flows.  In this paper we initiate
a
similar strategy for understanding the Yang-Mills flow.  We first recall some
basic setup and fundamental results about this flow.  In \cite{Rade} Rade showed
the smooth long time existence and convergence of Yang-Mills flow in dimensions
$n=2,3$.  Next, in \cite{Struwe}, Struwe gave a criterion for singularity
formation in dimension $n=4$ in terms of energy concentration, although it is to
date an open problem whether or not such energy concentration occurs.  Explicit
finite time singularities of Yang-Mills flow in dimensions $5 \leq n \leq 9$
were constructed in \cite{Gastel}.  Following \cite{CM}, we seek to 
investigate what a ``stable,'' or ``generic,'' singularity of Yang-Mills flow
looks like.  Similar constructions for the harmonic map flow were considered in
\cite{Zhang}.

In \cite{Ham2},
Hamilton defined an entropy functional for Yang-Mills flow akin to Huisken's
monotonicity formula for mean curvature flow \cite{Huis}, and Struwe's
monotonicity for harmonic maps \cite{Struwe2}, which is monotone against special
background manifolds (with an easily controlled decay rate in general).  In
\cite{Weinkove} Weinkove used this monotonicity formula to establish that type I
singularities of Yang-Mills flow admit blowup limits which are shrinking soliton
solutions of Yang-Mills flow (Definition \ref{def:soliton}).  At this point,
one may ask whether or not \emph{all} singularities of Yang-Mills flow admit
shrinking soliton blowup limits.  The corresponding statement for mean curvature
flow was established in work of Ilmanen, White \cite{Ilmanen, White}.  Because
of this it is reasonable to
initiate an in-depth study of mean curvature flow shrinkers to define the notion
of a stable singularity.  Despite the fact that it is unknown if all Yang-Mills
flow singularities in dimension $n \geq 5$ can be described by shrinkers,
we will nonetheless use them as our models to define a notion of stable
singularity.

In particular, we draw inspiration from Colding-Minicozzi's approach to
Huisken's monotonicity formula directly and
explicitly include a base point in the definition of Hamilton's Yang-Mills
entropy.  In particular, we will set
\begin{align*}
 \mathcal F_{x_0,t_0}(\N) = t_0^2 (4 \pi t_0)^{-\frac{n}{2}} \int_{\mathbb R^n}
\brs{F_{\N}}^2 e^{-\frac{\brs{x - x_0}^2}{4 t_0}} dV
\end{align*}
This functionals have the key property that their critical points are
self-shrinking solutions to Yang-Mills flow.  Moreover, the the entropy takes
the supremum of $\mathcal F$ over points in spacetime, specifically
\begin{align*}
 \gl(\N) = \sup_{x_0,t_0} \FF_{x_0,t_0}(\N).
\end{align*}
This quantity has many invariance properties, and is moreover monotone along a
solution to Yang-Mills flow (Proposition \ref{solitonmonotonicity}).  However,
it does not depend smoothly on $\N$, and thus is difficult to work with
analytically.  Nonetheless, we are able to relate stability with respect to $\N$
to the more calculationally tractable stability coming from the
$\FF$-functionals (see Definition \ref{stabdefn}).  This stability comes from a
more traditional second variation analysis for the functional.

In analyzing the $\FF$-functionals, in parallel to
\cite{CM}, we observe that two negative eigenforms for this second variation are
always present, corresponding to the Yang-Mills flow direction and also
translation in space.  Taking these explicitly into account yields an analytic
characterization of the formal definition of the $\FF$-stability of a shrinking
soliton (see Definition \ref{stabdefn}, Theorem \ref{stabthm}).  Also, as a
consequence of this analysis of the second variation, we establish a basic gap
theorem for solitons.

\begin{thm} \label{gapthm} A soliton satisfying $\brs{F_{\N}} \leq \frac{3}{8}$
is flat.
\end{thm}

Putting this discussion together and exploiting a number of interesting
properties of the $\FF$-functionals, we are able to relate $\FF$ stability with
entropy stability.  This theorem is analogous to
(\cite{CM} Theorem 0.15).
\begin{thm} \label{entequiv} Suppose $\N \in \mathfrak{S}$ is non cylindrical
with polynomial curvature growth. If $\N$ is $\mathcal{F}$-unstable then there
is a compactly supported variation $\N_s$ such that $\N_0 = \N$ and for all $s
\neq 0$,
\begin{equation}
\la(\N_s) < \la(\N). 
\end{equation}
\end{thm}

We note here that the only known examples of type I blowups (\cite{Grotowski,
Naito}) and shrinking solitons (\cite{Gastel}) are in dimensions $n \geq 5$,
whose entropies we compute in \S \ref{gastelsec}.  We will also show that this
dimensional restriction is necessary.  In particular,
while type I blowup limits are shrinking solitons, we show that any shrinking
soliton on $\mathbb R^4$ is automatically flat, thus ruling out type I blowups.

\begin{prop}\label{fourdprop} Let $E \to (M^4, g)$ be a smooth vector bundle,
and suppose $\N_t$
is a solution to Yang-Mills flow on $E$.  If $T < \infty$ is the maximal
existence time of the solution, then
\begin{align*}
\lim_{t \to T} (T - t) \brs{F_{\N_t}} = \infty.
\end{align*}
Moreover, any shrinking soliton on $\mathbb R^n$, $n \leq 4$, is flat.
\end{prop}

For this reason in four-dimensions it is more natural to define stability using
the second variation of
the Yang-Mills
functional itself, not the entropy.  This notion was introduced in \cite{Bourg},
where the second
variation of $\mathcal{YM}$ is analyzed in depth.  Among many results in that
paper is a strong rigidity result showing that for $SU(2), SU(3)$ or
$SO(4)$-bundles over $S^4$, the only stable Yang-Mills connections are either
self-dual or antiself-dual.  However, a recent result of Waldron \cite{Waldron}
suggests that blowup limits of finite time singularities of Yang-Mills flow in
dimension $n=4$ are \emph{not} instantons, and hence should not be stable.  This
suggests that it may be possible to construct smooth long time Yang-Mills flows
``generically" for vector bundles on four-manifolds with small gauge group.

Given the results above, a number of natural questions emerge following
parallel lines of thought from mean curvature flow.  First, we note that for
solutions to mean curvature flow it was shown in \cite{Ilmanen, White} that
arbitrary
singularities admits blowup limits which are shrinking solitons.  A very natural
question is whether the result of \cite{Weinkove} can be extended to show the
analogous
statement for Yang-Mills flow.  In other words, do arbitrary singularities of
Yang-Mills flow admit shrinking soliton blowup limits?  Another basic issue is
to construct more examples of shrinking solitons.  Despite the lack of examples,
one would still like to know the answers to some simple questions.  For
instance, for a given gauge group, what
is the minimal entropy shrinker?  Also, is it possible to classify stable
shrinkers?

Here is an outline of the rest of this paper.  In \S \ref{Ffunc} we recall
Hamilton's general entropy functional and monotonicity formula for Yang-Mills
flow, as well as a specialized version of this functional on $\mathbb R^n$
adapting Huisken's entropy for mean curvature flow to Yang-Mills flow.  In \S
\ref{vars} we establish variational properties of the entropy, leading to the
some corollaries on the structure of self-shrinkers of Yang-Mills
flow.  Next, in \S \ref{stability} we make some observations about the spectrum
of the second variation of $\FF$, leading to a characterization of
$\FF$-stability and the proof of Theorem \ref{gapthm}.  In \S
\ref{ss:Ffunctnentropy} we prove Theorem \ref{entequiv}, and we conclude 
in \S \ref{gastelsec} by computing the entropy of the Gastel shrinkers.

\vskip 0.1in

\textbf{Acknowledgements:} The authors would like to thank Michael Struwe for
his comments on an earlier draft of this paper.

\section{\texorpdfstring{$\mathcal F$}{F}-functional and entropy} \label{Ffunc}
\subsection{Background}
Let $(M^n,g)$ be a smooth, compact Riemannian manifold without boundary. Given a
vector bundle $E$ over $M$, let $S(E)$ denote the sections of $E$. For each
point $p$ choose a local basis of $TM$ given by $\{ \del_i \}$ with dual
elements $\{ dx^i \}$, where $dx^i(\del_j) = \delta_j^{i}$. Additionally, choose
a local basis for $E$ given by $\{ \mu_{\ga} \}$ with dual elements $\{
\mu_{\ga}^* \}$ where $\mu_{\ga}^*(\mu_{\gb}) = \delta_{\ga}^{\gb}$.
Given a chart containing $p \in M$ the action of a connection $\nabla$ on $E$ is
given in a local basis by
\begin{align*}
\nabla \mu_{\gb} &= \gG_{i \gb}^{\gd} dx^i \otimes \mu_{\gd}.
\end{align*}
Set $\Gamma  = ( \gG_{i \gb}^{\gd} dx^i \otimes \mu_{\delta} \otimes
\mu_{\beta}^* )$ to be the \textit{connection coefficient matrix} (associated
with $\nabla$) with respect to the basis.  The set of all connections over $M$
will be denoted by $\mathcal{A}_{E}$.
The actions of  $\nabla$ are extended to $TM$ by coupling it with the unique
Levi-Civita connection of $(M,g)$, given locally via
\begin{align*}
\nabla \del_j &= \left( \gG^{LC} \right)_{ij}^{k} dx^i \otimes \del_k.
\end{align*}
The actions of $\nab$ may be extended to tensorial combinations of $T^*M$ and
$E$ as well as their dual spaces. We let $\nab^*$ denote the formal adjoint of
$\nab$ with respect to the inner product.

Let $D$ be the \textit{exterior derivative}, or skew symmetrization of $\nabla$
over the tensor products of $T^*M$. Set $\Lambda^p(E) = \Lambda^p(M) \otimes
S(E)$.
 We let $D^{(p)}$ be the covariant connection from $\Lambda^p(E)$ to
$\Lambda^{p+1}(E)$, where the $p$ index will be dropped when understood. The
curvature tensor $F_{\nab} := D^{(1)} \circ D^{(0)} : \Lambda^0(E) \rightarrow
\Lambda^2(E)$ is given in local coordinates by
\begin{equation}\label{eq: Fcoords}
F_{\nab}= \left(  \del_i \gG_{j \ga}^{\gb}  - \del_j \gG_{i \ga}^{\gb}  - \gG_{i
\ga}^{\gd} \gG_{j \gd}^{\gb} + \gG_{j \ga}^{\gd} \gG_{i \gd}^{\gb}  \right) dx^i
\wedge dx^j \otimes \mu_{\gb} \otimes \mu_{\ga}^*. 
\end{equation}
Three more operators will be particularly important to our study. We set
$D^*_{\nab} := \nab^*$, which is a rescaled version of the formal $L^2$ adjoint
of $D$, chosen for computational convenience.
The \textit{Hodge Laplacian} is given by
\[
\lap_{D_{\nab}} : \Lambda^p(E) \rightarrow \Lambda^p(E) : \omega \mapsto
(D_{\nab}^* D_{\nab} + D_{\nab} D_{\nab}^*)\omega,
\]
while the \textit{rough (Bochner) Laplacian} and is given by
\[
\lap : \Lambda^p(E) \rightarrow \Lambda^p(E) : \gw \mapsto - \nabla^* \nabla
\gw.
\]
We next discuss background material pertinent to the study of the Yang-Mills
functional and its generalizations. We define the \emph{Yang-Mills functional}
by
\begin{equation}\label{eq:YM}
\mathcal{YM}(\nab) := || F_{\nab} ||_{L^2}^2 = \int_M{| F_{\nab}|^2 dV_g}.
\end{equation}
By computing the Euler Lagrange equation of \eqref{eq:YM} one may generate the
corresponding \emph{Yang-Mills flow} defined as follows.
\begin{equation}\label{eq:YMsystem}
\frac{\del \nab_t}{\del t} = - D_{\nab_t}^* F_{\nab_t}.
\end{equation}
Let $J$ and $K$ be multiindices of lengths $p_J-1$ and $p_K-1$ respectively,
where $J := (j_i)_{i=1}^{p_J -1}$ and $K := (k_i)_{i=1}^{p_K -1}$.  The
operation \textit{pound} is given by
\begin{align*}
\# : & (T^*M)^{\ten p_J} \ten \End(E) \times (T^*M)^{\ten p_K} \ten \End(E)
\rightarrow (T^*M)^{p_J + p_K - 2} \ten \End(E) : (A,B) \mapsto A \# B, \\
&(A \# B)(\del_{j_1}, ..., \del_{j_{p_J- 1}}, \del_{k_1}, ... , \del_{k_{p_K -
1}}) := \sum_{i=1}^nA(\del_{i}, \del_{j_1}, ..., \del_{j_{p_J - 1}}) B(
\del_i,\del_{k_1}, ... , \del_{k_{p_K - 1}}).
\end{align*}
In coordinates this is written in the form $ (A \# B)_{JK \ga}^{\gb} =
g^{jk}A_{j J \gd}^{\gb}B_{k K \ga}^{\gd}$.
Roughly speaking, $\#$ is matrix multiplication combined with contraction of the
first two forms.

Lastly we define the \textit{pound bracket} by 
\begin{align*}
[\cdot,\cdot]^{\#} : & (T^*M)^{\ten p_J} \ten \End(E) \times (T^*M)^{\ten p_K}
\ten \End(E) \rightarrow (T^*M)^{p_J + p_K - 2} \ten \End(E) \\
&: (A,B) \mapsto A\# B - B \# A.
\end{align*}

\begin{lemma}\label{D*D*2form}
Given $\N$ a connection and $\omega \in \Lambda^2(\End E)$, one has
\begin{equation}\label{eq:D*D*2formeq}
D^*D^* \omega = \frac{g^{i \ell}g^{jk}}{2} \left( F^{\gb}_{ij \gd}
\omega^{\gd}_{k \ell \ga} - F^{\gd}_{ij \ga} \omega_{k \ell \gd}^{\gb} \right),
\end{equation}
and in particular $D^*D^*F = 0$.

\begin{proof}
We compute
\begin{align*}
g^{i \ell}g^{jk} \nab_i \nab_j \omega_{k \ell \ga}^{\gb} &=\frac{g^{i
\ell}g^{jk}}{2}  \left( \nab_i \nab_j \omega_{k \ell \ga}^{\gb} - \nab_i \nab_j
\omega_{\ell k \ga}^{\gb} \right)\\
 &= \frac{g^{i \ell}g^{jk}}{2} [\nab_i, \nab_j] \omega_{k \ell \ga}^{\gb} \\
&= \frac{g^{i \ell}g^{jk}}{2} \left( \Rm_{ij k}^p \omega_{p \ell \ga}^{\gb} +
\Rm_{ij \ell}^p \omega_{k p \ga}^{\gb} - F^{\gd}_{ij \ga} \omega_{k \ell
\gd}^{\gb} + F^{\gb}_{ij \gd} \omega^{\gd}_{k \ell \ga} \right)\\
&= \frac{g^{i \ell}g^{jk}}{2} \left( g^{pq} \left( \Rm_{ji \ell q} \omega_{p k
\ga}^{\gb} + \Rm_{ij \ell q} \omega_{k p \ga}^{\gb} \right) - F^{\gd}_{ij \ga}
\omega_{k \ell \gd}^{\gb} + F^{\gb}_{ij \gd} \omega^{\gd}_{k \ell \ga} \right)\\
&= g^{i \ell}g^{jk} \left( g^{pq} \left( \Rm_{ij \ell q} \omega_{k p \ga}^{\gb}
\right) - \tfrac{1}{2} \left( F^{\gd}_{ij \ga} \omega_{k \ell \gd}^{\gb} +
F^{\gb}_{ij \gd} \omega^{\gd}_{k \ell \ga} \right) \right).
\end{align*}
Using the symmetries of the curvature tensor and $\gw$ it follows that $g^{i
\ell}g^{jk} g^{pq} \Rm_{ij \ell q} \omega_{k p \ga}^{\gb} = 0$.  Thus we
conclude that
\begin{align*}
g^{i \ell}g^{jk} \nab_i \nab_j \omega_{k \ell \ga}^{\gb} &= \frac{g^{i
\ell}g^{jk}}{2} \left( F^{\gd}_{ij \ga} \omega_{k \ell \gd}^{\gb} + F^{\gb}_{ij
\gd} \omega^{\gd}_{k \ell \ga} \right).
\end{align*}
Thus (\ref{eq:D*D*2formeq}) follows, from which the claim $D^* D^* F = 0$
immediately follows.
\end{proof}
\end{lemma}

\subsection{Hamilton's monotonicity}

In \cite{Ham2, Struwe, Chen} entropy functionals were defined which are monotone
along Yang-Mills flow.  This entropy involves integrating the density
$\brs{F_{\N}}^2$ against a solution to the backwards heat equation.  In what
follows we rederive this monotonicity formula.  For concreteness, given a
solution to Yang-Mills flow on $[0,T)$, and a final value $G_T$ we consider a
one-parameter family
\begin{equation}\label{eq:bkqrdhteq}
\begin{cases}
\frac{\del G}{\del t} &= - \lap G\\
G(T) & = G_T. 
\end{cases}
\end{equation}
As usual we will frequently let $G_T$ be a Dirac delta mass centered at some
point of interest.

\begin{lemma}\label{varintFk} Let $\N_t$ be a solution to Yang-Mills flow and
$G_t \in C^{\infty}(M)$ a solution to \eqref{eq:bkqrdhteq}.  The following
equality holds:
\begin{align}
\begin{split}\label{eq:varintFk}
&\frac{\del}{\del t} \left[ |F|^2 G \right] + 4 \left| \frac{\N G \hook F}{G} -
D^*F \right|^2 G  +4 \nab^* X_G(\nab)\\
& \hsp - 4 g^{ip}g^{jq}g^{rs} F_{pr \ga}^{\gb} F_{qs \gb}^{\ga} \left(( \nab_i
\nab_j G) - \frac{(\nab_i G)(\nab_j G)}{G} \right) = 0.
\end{split}
\end{align}
where $X_G(\nab) := \frac{1}{4} |F|^2 (\nab G) + (\N G \hook F ) \# F - D^*F \#
F$.

\begin{proof}
Differentiating $|F|^2 G$ yields that
\begin{equation*}
\frac{\del}{\del t}\left[ |F|^2 G \right] = \frac{\del}{\del t} \left[ |F|^2
\right] G  +  |F|^2 \left( \frac{\del G}{\del t} \right).
\end{equation*}
For the first term on the right, since $\nab$ satisfies Yang-Mills flow,
$\dot{\gG}_{j \gb}^{\ga} = g^{uv}\nab_u F_{v j \gb}^{\ga}$, so
we compute, while incorporating divergence terms, the quantity
\begin{align*}
(\del_t |F|^2) G  &=  2 \langle \del_t F , F \rangle G  \\
&= 2 \langle D \dot{\gG}, F \rangle G \\
&= - 2 g^{ip} g^{jq} (D_i \dot{\gG}_{j \ga}^{\gb}) F_{pq \gb}^{\ga} G \\
&=  -2 g^{ip} g^{jq}\left( (\nab_i \dot{\gG}_{j \ga}^{\gb}) F_{pq \gb}^{\ga} -
(\nab_j \dot{\gG}_{i \ga}^{\gb}) F_{pq \gb}^{\ga} \right) G \\
&= -4 g^{ip} g^{jq} (\nab_i \dot{\gG}_{j \ga}^{\gb}) F_{pq \gb}^{\ga} G \\
& = -4 g^{ip}g^{jq} \left( \nab_i \left( \dot{\gG}_{j \ga}^{\gb} F_{pq
\gb}^{\ga} G \right) - \dot{\gG}_{j \ga}^{\gb} (\nab_i F_{pq \gb}^{\ga}) G -
\dot{\gG}_{j \ga}^{\gb} F_{pq \gb}^{\ga} (\nab_i G )\right)\\
& = 4 g^{ip}g^{jq} g^{vw}  \left( (\nab_v F_{w j \ga}^{\gb}) (\nab_i F_{pq
\gb}^{\ga}) G + (\nab_v F_{w j \ga}^{\gb}) F_{pq \gb}^{\ga} (\nab_i G ) \right)
\\
& \hsp -4 g^{ip}g^{jq} g^{vw} \left( \nab_i \left( (\nab_v F_{w j \ga}^{\gb})
F_{pq \gb}^{\ga} G \right)\right). 
\end{align*}
For the next term we have, using the second Bianchi identity and multiple
insertions of divergence terms,
\begin{align*}
 | F |^2 \del_t G
&=  - g^{i p} g^{j q} F_{ij \ga}^{\gb} F_{pq \gb}^{\ga}(- \lap G)\\
&= g^{i p} g^{j q} F_{ij \ga}^{\gb} F_{pq \gb}^{\ga} ( g^{vw}\nab_{v} \nab_{w} 
G) \\
&= g^{i p} g^{j q} g^{vw}  \left(\nab_v \left( F_{ij \ga}^{\gb} F_{pq \gb}^{\ga}
( \nab_{w}  G) \right) -2  (\nab_v F_{ij \ga}^{\gb}) F_{pq \gb}^{\ga}  (\nab_{w}
G) \right) \\
&= g^{i p} g^{j q} g^{vw}  \left(\nab_v \left( F_{ij \ga}^{\gb} F_{pq \gb}^{\ga}
( \nab_{w}  G) \right) + 2  (\nab_i F_{jv \ga}^{\gb} + \nab_j F_{vi \ga}^{\gb})
F_{pq \gb}^{\ga}  (\nab_{w} G) \right) \\
&= g^{i p} g^{j q} g^{vw}  \left(\nab_v \left( F_{ij \ga}^{\gb} F_{pq \gb}^{\ga}
( \nab_{w}  G) \right) + 4  (\nab_i F_{jv \ga}^{\gb}) F_{pq \gb}^{\ga} 
(\nab_{w} G) \right) \\
&= g^{i p} g^{j q} g^{vw}  \left(\nab_v \left( F_{ij \ga}^{\gb} F_{pq \gb}^{\ga}
( \nab_{w}  G) \right) + 4 \nab_i\left( F_{jv \ga}^{\gb} F_{pq \gb}^{\ga} 
(\nab_{w} G) \right) \right)\\
& \hsp -4  g^{i p} g^{j q} g^{vw}  \left( F_{jv \ga}^{\gb}(\nab_i  F_{pq
\gb}^{\ga})  (\nab_{w} G) + F_{jv \ga}^{\gb}  F_{pq \gb}^{\ga}  (\nab_i \nab_{w}
G) \right).  
\end{align*}
We combine the identities and sort out the divergence terms ($\mathsf{Div}$)
with some reindexing,
\begin{align*}
& \mathsf{Div} = g^{i p} g^{j q} g^{vw}  \left(\nab_v \left( F_{ij \ga}^{\gb}
F_{pq \gb}^{\ga} ( \nab_{w}  G) \right) + 4 \nab_i\left( F_{jv \ga}^{\gb} F_{pq
\gb}^{\ga}  (\nab_{w} G) \right)+ 4  \left( \nab_i \left( (\nab_w F_{jv
\ga}^{\gb}) F_{pq \gb}^{\ga} G \right) \right) \right)\\
 &= g^{i p} g^{j q} g^{vw} \left(\nab_v \left( F_{ij \ga}^{\gb} F_{pq \gb}^{\ga}
( \nab_{w}  G) \right) + 4 \nab_v \left( F_{ji \ga}^{\gb} F_{wq \gb}^{\ga} 
(\nab_{p} G) \right)+ 4  \left( \nab_v \left( (\nab_p F_{ji \ga}^{\gb}) F_{wq
\gb}^{\ga} G \right) \right) \right)\\
 & = g^{i p} g^{j q} g^{vw} \left(\nab_v \left( F_{ij \ga}^{\gb} F_{pq
\gb}^{\ga} ( \nab_{w}  G) \right) + 4 \nab_v \left( F_{ij \ga}^{\gb} F_{qw
\gb}^{\ga}  (\nab_{p} G) \right)+ 4  \left( \nab_v \left( (\nab_p F_{ij
\ga}^{\gb}) F_{qw \gb}^{\ga} G \right) \right) \right).
\end{align*}
Therefore in coordinate invariant form,
\begin{align*}  
\mathsf{Div} & = -4 \nab^* \left(- \frac{1}{4} |F|^2 \nab G + (\N G \hook F) \#
F - D^*F \# F \right).
\end{align*}
We set $X_G(\nab) := -\frac{1}{4} |F|^2 (\nab G) + (\N G \hook F) \# F - D^*F \#
F$. Combining all terms we have
\begin{align*}
\tfrac{\del}{\del t} \left[ |F|^2 G \right]
& =   4 g^{ip} g^{jq} g^{uv} \left( (\nab_u F_{v j \gb}^{\ga} )( \nab_i F_{pq
\ga}^{\gb}) G + 2 (\nab_u F_{v j \gb}^{\ga} )F_{pq \ga}^{\gb} (\nab_i G) + F_{vj
\ga}^{\gb} F_{pq \gb}^{\ga}  \left( \nab_i \nab_{u} G \right) \right) - 4 \nab^*
X_G(\nab)\\
&= -4 \left|D^* F\right|^2 + 8 \left\langle D^*F, \nab G \hook F \right\rangle +
4 g^{ip} g^{jq} g^{uv} F_{vj \ga}^{\gb} F_{pq \gb}^{\ga}  \left( \nab_i \nab_{u}
G \right) - 4( \nab^* X_G(\nab)).
\end{align*}
We recombine terms and observe
\begin{equation*}
\left| \frac{\N G \hook F}{G} -D^*F \right|^2 = \left| \frac{(\nab G) \hook
F}{G} \right|^2 - 2 \left\langle D^*F ,  \frac{\N G \hook F}{G} \right\rangle +
\left| D^*F \right|^2.
\end{equation*}
Therefore we incorporate this in and have
\begin{align*}
\frac{\del}{\del t} \lbr \left| F \right|^2 G \rbr &=-  4 \left| \frac{\N G
\hook F}{G} - D^*F \right|^2 + 4 \left| \frac{(\nab G) \hook F}{G} \right|^2\\
& \hsp + 4 g^{ip} g^{jq} g^{uv} F_{vj \ga}^{\gb} F_{pq \gb}^{\ga}  \left( \nab_i
\nab_{u} G \right) - 4( \nab^* X_G(\nab)).
\end{align*}
The result follows.
\end{proof}
\end{lemma}

\begin{cor}\label{translatortime} Let $\N_t$ be a solution to Yang-Mills flow 
and $G_t \in C^{\infty}(M)$ a solution to \eqref{eq:bkqrdhteq}.  The following
equality holds:
\begin{align*}
&\frac{\del}{\del t} \left[\int_M{ |F|^2 G dV_g} \right] + 4 \int_M{ \left|
\frac{\N G \hook F}{G} - D^*F \right|^2 G dV_g} \\
&\hsp  - 4 \int_M{ g^{ip}g^{jq}g^{rs} F_{pr \ga}^{\gb} F_{qs \gb}^{\ga} \left((
\nab_i \nab_j G) - \frac{(\nab_i G)(\nab_j G)}{G} \right) dV_g} = 0.
\end{align*}
\end{cor}

\begin{cor}[Hamilton's Entropy Monotonicity Formula, \cite{Ham2} Theorem C]
\label{hamentromon}  Let $\N_t$ be a solution to Yang-Mills flow and $G_t
\in C^{\infty}(M)$ a solution to \eqref{eq:bkqrdhteq}.  Then
\begin{align}
\begin{split}\label{eq:hamjunkterm}
0 &= \frac{\del}{\del t} \left[ (T-t)^2 \int_M{|F|^2 G dV_g} \right] + 4 (T-t)^2
\int_M{\left| \frac{\N G \hook F}{G} - D^*F \right|^2 G dV_g} \\
& \hsp - 4(T-t)^2 \int_M{g^{ip}g^{jq}g^{rs} F_{pr \ga}^{\gb} F_{qs \gb}^{\ga}
\left( (\nab_i \nab_j G) - \frac{(\nab_i G)(\nab_j G)}{G} + \frac{ G
{g_{ij}}}{2(T-t)} \right)dV_g}. 
\end{split}
\end{align}
\end{cor}

This monotonicity formula is used in proving the following result of Hamilton: 

\begin{thm}[Hamilton's Monotonicity Formula, \cite{Ham2} Theorem
C]\label{thm:hammonoform}
Given the functional
\begin{equation*}
\mathcal{F}(\N,t) :=  (T-t)^2 \int_{M} |F_{\N}|^2 G dV_g,
\end{equation*}
suppose $\N_t$ is a solution to Yang-Mills flow on $t \in [0,T)$. Then
$\mathcal{F}(\N_t,t)$ is monotone decreasing in $t$ when $M$ is Ricci parallel
with weakly positive sectional curvatures, while on a general manifold
\begin{equation*}
\mathcal{F}(\N_t,t) \leq C_M \mathcal{F}(\N_{\tau},\tau) + C_M(t-\tau)^2
\mathcal{YM}(\N_0). 
\end{equation*}
whenever $T-1 \leq \tau \leq t \leq T$, and $C_M$ is a constant depending only
on $M$.
\end{thm}

\subsection{Monotone entropy functionals on \texorpdfstring{$\mathbb
R^n$}{Rn}}\label{monentropRn}

In this subsection we specialize Hamilton's monotonicity formula \cite{Ham2} to
the case of $\mathbb R^n$ (compare \cite{Chen,Naito}).  We also observe the
existence
of ``steady'' and ``expanding'' entropy functionals which are fixed on steady
and expanding solitons respectively.  These functionals are so far only formal
objects, as they involve integrals which are not likely to converge in general. 
We will verify the corresponding monotonicity formulas in Proposition
\ref{monform}.

\begin{defn}\label{def:htkerFfunct} Let $M = \mathbb{R}^n$, $\nab$ a connection,
and $x_0 \in \mathbb{R}^n $. The \textit{shrinker kernel based at $(x_0,t_0)$}
is given by, for $t < t_0$,
\begin{equation}\label{eq:shrker}
G_{x_0,t_0}(x,t) := \frac{e^{\frac{- |x - x_0 |^2}{4(t_0 -t)}}}{(4 \pi (t_0 -
t))^{n/2}},
\end{equation}
and the \emph{$\mathcal{F}$-functional} is given by
\begin{equation}\label{eq:monforms}
 \mathcal{F}_{x_0,t_0}(\nab , t) := (t- t_0)^2 \int_{\bRn}{|F_{\nab}|^2
G_{x_0,t_0}(x,t) dV}.
\end{equation}
The \textit{translator kernel based at $(x_0,t_0)$} is given by, for $t, t_0 \in
\mathbb{R}$
\begin{equation}\label{eq:trnker}
G^{\mathcal{T}}_{x_0,t_0}(x,t) := e^{\langle x_0, x \rangle - |x_0|^2 (t-t_0)},
\end{equation}
and the \emph{$\mathcal{F}^{\mathcal{T}}$-functional} will be given by
\begin{equation}\label{eq:monformt}
\mathcal{F}_{x_0,t_0}^{\mathcal{T}}(\nab , t) := \int_{\bRn}{|F_{\nab}|^2
G_{x_0,t_0}^{\mathcal{T}}(x,t)dV}.
\end{equation}
The \textit{expander kernel based at $(x_0,t_0)$} is given by, for $t >t_0$,
\begin{equation}\label{eq:expker}
G^{\mathcal{E}}_{x_0,t_0}(x,t) := \frac{e^{\frac{|x-x_0|^2}{4(t-t_0)}}}{(4 \pi
(t-t_0))^{n/2}},
\end{equation}
and the \emph{$\mathcal{F}^{\mathcal{E}}$-functional} will be given by
\begin{equation}\label{eq:monforme}
\mathcal{F}_{x_0,t_0}^{\mathcal{E}}(\nab, t) := (t_0 - t)^2
\int_{\bRn}{|F_{\nab}|^2 G_{x_0,t_0}^{\mathcal{E}}(x,t) dV}.
\end{equation}
\end{defn}

\begin{prop}[Monotonicity formulas]\label{monform}  Let $\ga,\gb \in [-\infty,
\infty]$ with $\ga < \gb$, and let $\nab_t \in \mathcal{A}_E \times [\ga,\gb)$
be a solution to Yang-Mills flow. Given $(x_0,t_0) \in \mathbb R^n \times
[\ga,\gb]$, the functionals $\mathcal{F}_{x_0,t_0}(\nab_t,t)$,
$\mathcal{F}^{\mathcal{T}}_{x_0,t_0}(\nab_t,t)$,
$\mathcal{F}^{\mathcal{E}}_{x_0,t_0}(\nab_t,t)$ are monotonically decreasing in
$t$.

\begin{proof}
We first demonstrate that $G, G^{\mathcal T}$ and $G^{\mathcal E}$ all satisfy
(\ref{eq:bkqrdhteq}). First we have, for the shrinker kernel \eqref{eq:shrker},
\begin{align*}
\left( \lap + \frac{\del}{\del t} \right) \left[ G_{x_0,t_0}(x,t) \right] &=
\frac{1}{(4 \pi (t_0-t))^{\frac{n}{2}}} \nab_i \left(
-\frac{(x-x_0)_i}{2(t_0-t)} e^{\frac{-|x - x_0|^2}{4(t_0-t)}} \right) \\
& \hsp -  G_{x_0,t_0}(x,t) \left( \frac{|x - x_0|^2}{4(t_0-t)^2} -
\frac{n}{2(t_0-t)} \right)\\
& = 0.
\end{align*}
This verifies the first case. We next compute, for the translator kernel
\eqref{eq:trnker},
\begin{align*}
\left( \lap + \frac{\del}{\del t} \right) \left[ G^{\mathcal{T}}_{x_0,t_0}(x,t)
\right] = \nab_i \left( x_i G^{\mathcal{T}}_{x_0,t_0}(x,t) \right) - |x_0|^2
G^{\mathcal{T}}_{x_0,t_0}(x,t) = 0.
\end{align*}
This verifies the second case. Finally we consider the expander kernel
\eqref{eq:expker},
\begin{align*}
\left( \lap + \frac{\del}{\del t} \right) \left[ G^{\mathcal{E}}_{x_0,t_0}(x,t)
\right] & = \frac{1}{(4 \pi (t - t_0))^{\frac{n}{2}}} \nab_i \left( \frac{
(x-x_0)_i}{2(t-t_0)} e^{\frac{|x - x_0|^2}{4(t-t_0)}} \right) \\
& \hsp - G^{\mathcal{E}}_{x_0,t_0}(x,t) \left( \frac{n}{2(t-t_0)} + \frac{|x -
x_0|^2}{4(t-t_0)^2} \right)\\
& = 0.
\end{align*}
We apply Corollary \ref{hamentromon}, and show the vanishing of the last term of
 \eqref{eq:hamjunkterm}.  It suffices to compute, for $\mathcal{X} \in \{
\mathcal{E}, \mathcal{T} \}$ and omitted, quantities of the form
\begin{equation}
\nab_i \nab_j \left( G^{\mathcal{X}}_{x_0,t_0} \right) - \frac{(\nab_i
G^{\mathcal{X}}_{x_0,t_0})( \nab_j G^{\mathcal{X}}_{x_0,t_0}
)}{G^{\mathcal{X}}_{x_0,t_0}}.
\end{equation}
In each case, we compute the two main quantities of the last quantity
of\eqref{eq:hamjunkterm}. We first address the shrinker kernel
\eqref{eq:shrker}.
\begin{align}
\begin{split}\label{eq:monforms1}
\nab_i \nab_j \left( G_{x_0,t_0} \right) &= \frac{1}{(4 \pi (t_0 -
t))^{\frac{n}{2}}}\nab_i \nab_j \lbr e^{-\frac{|x- x_0|^2}{4(t_0-t)} }\rbr\\
&= \frac{1}{(4 \pi (t_0-t))^{\frac{n}{2}}}\nab_i \lbr -\frac{(x-x_0)_j}{2(t_0 -
t)}e^{\frac{|x-x_0|^2}{4(t_0 - t)} }\rbr\\
&= \left( - \frac{\delta_{ij}}{2(t_0-t)} + \frac{(x-x_0)_i
(x-x_0)_j}{4(t_0-t)^2} \right) G_{x_0,t_0}.
\end{split}
\end{align}
We also have
\begin{equation}\label{eq:monforms2}
\frac{(\nab_i G_{x_0,t_0}) (\nab_j G_{x_0,t_0}) }{G_{x_0,t_0}} = \frac{(x-x_0)_i
(x-x_0)_j}{4(t_0-t)^2} G_{x_0,t_0}.
\end{equation}
Combining \eqref{eq:monforms1} and \eqref{eq:monforms2} yields
\begin{equation*}
\nab_i \nab_j (G_{x_0,t_0}) - \frac{(\nab_i G_{x_0,t_0}) (\nab_j
G_{x_0,t_0})}{G_{x_0,t_0}} = - \frac{\delta_{ij}}{2(t_0-t)} G_{x_0,t_0}.
\end{equation*}
We conclude that the last term of \eqref{eq:hamjunkterm} vanishes, and thus the
temporal monotonicity of $\mathcal{F}_{x_0,t_0}$ with respect to $t$ holds.

We next consider the translator kernel \eqref{eq:trnker}. Observe that
\begin{equation*}
\nab_i \nab_j \left( G^{\mathcal{T}}_{x_0,t_0} \right) - \frac{(\nab_i
G^{\mathcal{T}}_{x_0,t_0})( \nab_j
G^{\mathcal{T}}_{x_0,t_0})}{G^{\mathcal{T}}_{x_0,t_0}} = \nab_i \left( (x_0)_j
G^{\mathcal{T}}_{x_0,t_0} \right)-  (x_0)_i (x_0)_j G^{\mathcal{T}}_{x_0,t_0}=
0.
\end{equation*}
It follows that the last term in the expression of Corollary
\ref{translatortime} vanishes, and so the monotonicity of
$\mathcal{F}_{x_0,t_0}^{\mathcal{T}}$ holds.  Lastly we consider the expander
kernel \eqref{eq:expker}. Observe that
\begin{align}
\begin{split}\label{eq:monforme1}
\nab_i \nab_j \left( G^{\mathcal{E}}_{x_0,t_0} \right) &= \frac{1}{(4 \pi
(t-t_0))^{\frac{n}{2}}} \nab_i \nab_j \lbr e^{\frac{|x-x_0|^2}{4(t-t_0)}} \rbr\\
&=\frac{1}{(4 \pi (t-t_0))^{\frac{n}{2}}} \nab_i \lbr \frac{
(x-x_0)_j}{2(t-t_0)} e^{\frac{|x-x_0|^2}{4(t-t_0)}} \rbr \\
&=\left( \frac{\delta_{ij}}{2(t-t_0)} + \frac{ (x-x_0)_i  (x-x_0)_j}{4(t-t_0)^2}
\right) G^{\mathcal{E}}_{x_0,t_0}.
\end{split}
\end{align}
We also have that
\begin{align}
\begin{split}\label{eq:monforme2}
\frac{(\nab_i G^{\mathcal{E}}_{x_0,t_0})(\nab_j
G^{\mathcal{E}}_{x_0,t_0})}{G^{\mathcal{E}}_{x_0,t_0}} &= \frac{(x-x_0)_i
(x-x_0)_j}{4(t-t_0)^2} G^{\mathcal{E}}_{x_0,t_0}.
\end{split}
\end{align}
Combining \eqref{eq:monforme1} and \eqref{eq:monforme2} yields
\begin{equation*}
\nab_i \nab_j (G^{\mathcal{E}}_{x_0,t_0}(x)) - \frac{(\nab_i
G^{\mathcal{E}}_{x_0,t_0}) (\nab_j
G^{\mathcal{E}}_{x_0,t_0})}{G^{\mathcal{E}}_{x_0,t_0}} =
\frac{\delta_{ij}}{2(t-t_0)} G^{\mathcal{E}}_{x_0,t_0},
\end{equation*}
We conclude the monotonicity of $\mathcal{F}_{x_0,t_0}^{\mathcal{E}}$ from
Corollary \ref{hamentromon}.
\end{proof}
\end{prop}
\subsection{Entropy and basic properties of shrinkers}

Provided the definition of the $\mathcal{F}$-functional \eqref{eq:monforms}
given in \S \ref{monentropRn}, we set $G_0 := (4 \pi t_0)^{-n/2}
e^{-\frac{|x-x_0|^2}{4t_0}}$, and have
\begin{equation*}
\mathcal{F}_{x_0,t_0}(\nab) :=  \mathcal{F}_{x_0,t_0}(\nab, 0) = t_0^2
\int_{\bRn}{|F_{\nab}|^2 G_0 dV}.
\end{equation*}

\begin{defn} For a connection $\N$ the \textit{entropy} is given by
\begin{equation*}
\la(\nab) = \sup_{x_0 \in \mathbb{R}^n, t_0 > 0} \mathcal{F}_{x_0,t_0}(\nab).
\end{equation*}
\end{defn}

\begin{defn} Let $\N_t$ be a smooth one-parameter family of connections on
$\mathbb R^n \times (- \infty, 0)$. Then $\N_t$ is a \emph{self-similar
solution} if
\begin{equation}\label{def:selfsimsol}
D_{\N_t}^*F_{\N_t} - \frac{x }{2t} \hook F_{\N_t}= 0.
\end{equation}
\end{defn}

\begin{prop}\label{prop:sssiffgtfm} Suppose $\N_t$ is a self-similar solution to
Yang-Mills flow on $\mathbb{R}^n \times (- \infty, 0 )$. Then there exists an
exponential gauge for $\N$ such that the connection coefficients satisfy 
\begin{equation}\label{eq:gGscalsqrtt}
\gG(x,t) = \tfrac{1}{\sqrt{-t}} \gG \left( \tfrac{x}{\sqrt{-t}}, -1 \right).
\end{equation}
\end{prop}

The exponential gauge is unique up to initial choice of frame at $x=0$.  This
proposition is a consequence of two lemmas from \cite{Weinkove}. For the first
statement we refer the reader directly to the text.

\begin{lemma}[end of Theorem 3.1 \cite{Weinkove}, pp.8]\label{equivDE} A
solution $\N_t$ to Yang-Mills flow is furthermore a solution to the differential
equation
\begin{equation}\label{eq:equivDE1}
\frac{\del \N_t}{\del t}  + \frac{x}{2t} \hook F_{\N_t} = 0,
\end{equation}
if and only if in an exponential gauge the connection matrices satisfy
\begin{equation}\label{eq:equivDE2}
\frac{\del \gG_t}{\del t} + \left( \frac{x}{2 t} \hook \del \gG_t \right) +
\frac{1}{2t} \gG_t =0.
\end{equation}
\end{lemma}

The second lemma demonstrates the equivalences of a characteristic scaling law
of connections with \eqref{eq:equivDE2}. We will include a detailed proof for
convenience.

\begin{lemma}[\cite{Weinkove} Lemma 3.2]\label{wnkovescallawsoliton} Let the
one-parameter family $\N_t$ with connection coefficient matrices $\gG_t$ be a
solution to Yang-Mills flow. Then $\N_t$ satisfies
\begin{equation}\label{selfsimseqde}
\frac{\del \gG_t}{\del t} + \left( \frac{x}{2 t} \hook \del \gG_t \right) +
\frac{1}{2t} \gG_t =0,
\end{equation}
on $\mathbb{R}^n \times (-\infty, 0)$, if and only if for all $\la \neq 0$,
\begin{equation}\label{solitonscallaw}
\gG (x,t) = \la \gG (\la x , \la^2 t).
\end{equation}

\begin{proof}
Assuming $\N$ satisfies \eqref{solitonscallaw}, we differentiate
\eqref{solitonscallaw} with respect to $\la$ and then evaluate at $\la = 1$:
\begin{align*}
0 & = \left. \frac{\del}{\del \la} \left(  \la \gG ( \la x , \la^2 t ) -
\gG(x,t) \right) \right|_{\la=1}\\
&= \left. \gG +  \la^2 x^k  \del_k \gG + 2 \la t \del_t \gG \right|_{\la = 1}\\
&=  \gG +  x^k  \del_k \gG + 2 t \del_t \gG.
\end{align*}
Dividing by $2t$ produces the desired result.
Next, we show that a solution of \eqref{selfsimseqde} must consequently satisfy
the scaling law \eqref{solitonscallaw}. To do so, we let $\widetilde {\N}$ be a
solution to Yang-Mills flow satisfying \eqref{solitonscallaw}, so that
$\widetilde{\N}$ is a solution to \eqref{selfsimseqde} with connection
coefficient matrix $\gG$, and let $\N$ be yet another solution to
\eqref{selfsimseqde} with connection coefficient matrix $\widetilde{\gG}$. Set
$\Upsilon_t := \gG_t - \widetilde{\gG}_t$. Note that for each $t$, $\Upsilon_t$
is in the kernel of the following operator
\begin{equation*}
\Phi : \Lambda^1(\End E) \to \Lambda^1(\End E) : B \mapsto \frac{\del B}{\del t}
+ \left( \frac{x}{2t} \hook \del B \right).
\end{equation*}
We first verify that for any $s \in (- \infty, 0)$ the hypersurface
$\mathbb{R}^n \times \{ s \}$ is non-characteristic with respect to the operator
$\Phi$. This is equivalent to showing that the symbol is non degenerate in the
transverse direction of the boundary of $\mathbb{R}^n \times \{ s \}$, that is,
that
\begin{equation*}
\langle \sigma[\Phi], \del_t \rangle \neq 0.
\end{equation*}
Given that
\begin{equation*}
(\sigma[\Phi](B))_{k \ga}^{\gb} = \xi_t B + \left( \tfrac{x}{2t} \hook \xi_{x} B
\right),
\end{equation*}
then we have that
\begin{equation*}
\langle \sigma[\Phi], \xi_t \rangle = |\xi_t|^2 \neq 0.
\end{equation*}
Thus, by Holmgren's Uniqueness Theorem (cf. \cite{Taylor} pg. 433), there exists
some $\epsilon >0$ such that on $\mathbb{R}^n \times [ s - \epsilon, s +
\epsilon]$, we have $\Phi(\Upsilon) = 0$.
This demonstrates openness of the set
\begin{equation*}
\mathcal{T} := \left\{ \theta \in (-\infty,0) : \Upsilon_{\theta} = 0 \right\}.
\end{equation*}
Since this set is closed (the inverse image of zero under a continuous map) by
the connectedness of $(- \infty, 0)$, then $\mathcal{T} =(- \infty, 0)$ so we
have $\gG_t = \widetilde {\gG}_t$, as desired. The result follows.
\end{proof}
\end{lemma}

\begin{proof}[Proof of Proposition \ref{prop:sssiffgtfm}]
Define $\N_t$ to be a family of connections which are furthermore solutions to
Yang-Mills flow with coefficient matrices which satisfy
\begin{equation}
\gG(x,t) := \left( \tfrac{1}{\sqrt{-t}} \right) \gG \left( 
\tfrac{x}{\sqrt{-t}}, -1 \right).
\end{equation}
We verify that $\gG(x,t)$ satisfies the scaling law \eqref{solitonscallaw} by
computation:
\begin{align*}
\la \gG( \la x, \la^2 t  ) & = \la \tfrac{1}{\sqrt{- \la^2 t}} \left( \tfrac{\la
x}{\sqrt{- \la^2 t}}, -1 \right)\\
&=  \tfrac{1}{\sqrt{-t}} \gG \left(  \tfrac{x}{\sqrt{-t}}, -1 \right) \\
&= \gG(x,t).
\end{align*}
Thus the scaling law holds. It therefore follows by Lemmas \ref{equivDE} and
\ref{wnkovescallawsoliton} that this is equivalent to $\N_t$ satisfying
\eqref{eq:equivDE1}. But since $\N_t$ is a solution to Yang-Mills flow, we
conclude that
\begin{equation}\label{eq:solitonequiv1}
D_{t}^*F_{t} = \frac{x}{2t} \hook F_{t}.
\end{equation}
The result follows.
\end{proof}

\begin{defn}\label{def:soliton} A connection $\nab$ is a \textit{soliton} if,
for all $x \in \bRn $,
\begin{equation}\label{eq:solitondef}
D^*_{\nab}F_{\nab} + \frac{x}{2} \hook F_{\nab} = 0.
\end{equation}
\end{defn}
This definition captures the notion of a self-similar solution by considering
the $t=-1$ slice.  As exhibited in \cite{Weinkove}, all type I singularities of
Yang-Mills flow admit blowup solutions which are nontrivial solitons, thus their
study is central to understanding singularity formation of the flow.  In this
section and the next we collect a number of observations concerning solitons and
their structure.  First, we observe that solitons can be interpreted as
Yang-Mills connections for a certain conformally modified metric on $\mathbb
R^n$.

\begin{prop} \label{confchange} Suppose $\N$ is a soliton on $\mathbb R^n, n
\geq 5$.  Then $\N$ is a Yang-Mills connection with respect to the metric
$g_{ij} = e^{- \frac{\brs{x}^2}{2(n-4)}} \gd_{ij}$.
\begin{proof} A calculation shows that for a connection $\N$, Riemannian metric
$g$ and function $\phi$, one has
\begin{align*}
- D^*_{e^{2\phi} g} F = e^{-2\phi} \left[ - D^*_{g} F + (n-4) \N \phi \hook F
\right].
\end{align*}
With the choice $\phi = - \frac{\brs{x}^2}{4(n-4)}$, comparing against
(\ref{eq:solitondef}) yields the result.
\end{proof}
\end{prop}

Next we establish a number of preliminary properties of solitons and Yang-Mills
flow blowups in preparation for understanding the variational properties of the
$\FF$-functional.  We will use these to show in Corollary \ref{Fcrit} below that
solitons are, after reparameterizing in space and time, the critical points of
the $\FF$-functional.  First though we consider a more general notion of
soliton.
 \begin{defn} The \textit{$(x_0,t_0)$-soliton operator} is given by
\begin{align*}
S_{x_0,t_0} :\ \mathcal{A}_E \to \Lambda^1(\End E) : \ \nab \mapsto
D_{\nab}^*F_{\nab} + \frac{(x - x_0) }{2t_0}\hook F_{\nab}.
\end{align*}
 \end{defn}

\begin{defn}\label{X0T0solitondef} Suppose $\N \in \ker S_{x_0,t_0}$ so that
\begin{equation}\label{eq:X0T0solitondef}
D^*_{\N}F_{\N} + \frac{(x-x_0)}{2 t_0} \hook F_{\N} = 0.
\end{equation}
Then $\N$ is called a \textit{$(x_0,t_0)$-soliton}.
\end{defn}

We next demonstrate the correspondence between the set of $(0,1)$-solitons,
denoted by $\mathfrak{S}$, and the set of $(x_0,t_0)$-solitons, denoted by
$\mathfrak{S}_{x_0,t_0}$.

\begin{lemma}\label{sssaresolitons}
For all $x_0 \in \mathbb{R}^n$ and $t_0 \in \mathbb{R}$, the sets
$\mathfrak{S}_{x_0,t_0}$ and $\mathfrak{S}$ are in bijective correspondence.

\begin{proof}
Beginning with $\nab \in \mathfrak{S}$, we set
\begin{align*}
\widetilde {\N} \left( x \right)  &:= \frac{1}{\sqrt{t_0}} \N\left(
\frac{x-x_0}{\sqrt{t_0}}\right).
\end{align*}
Computation yields
\begin{align*}
\left[ D_{\widetilde {\N}} F_{\widetilde {\N}} \right]_{x}
&= \frac{1}{t_0^{3/2}}\left[  D_{\N} F_{\N}\right]_{\frac{x - x_0}{\sqrt{t_0}}}
\\
&= \frac{1}{t_0^{3/2}} \left( \frac{x_0-x}{2 \sqrt{t_0}} \right) \hook \left[
F_{\N} \right]_{\frac{x - x_0}{\sqrt{t_0} }} \\
& =  \left( \frac{x_0-x}{2 t_0} \right) \hook\left( \frac{1}{t_0} \left[ F_{\N}
\right]_{\frac{x - x_0}{\sqrt{t_0}}} \right) \\
& =  \left( \frac{x_0-x}{2 t_0} \right) \hook \left[ F_{\widetilde {\N}}
\right]_{x}.
\end{align*}
Conversely given $\widetilde {\N} \in \mathfrak{S}_{x_0,t_0}$, we define $\N(x)
:=\sqrt{t_0} \widetilde {\N}\left( x_0 + \sqrt{t_0} x \right)$.  A similar
calculation shows that $\N \in \mathfrak{S}$.  The result follows.
\end{proof}
\end{lemma}

\subsection{Polynomial energy growth}

In the computations to follow deriving the first and second variation of entropy
we integrate by parts and encounter many quantities whose integrability is not
immediately clear.  For this reason we will add an extra condition to the
solitons we consider, namely that of ``polynomial energy growth,'' made precise
below.  We give a formal argument in Proposition \ref{pegprop} showing that
blowup limits of Yang-Mills flow, should automatically satisfy this hypothesis. 
Moreover, for the more delicate analytic arguments we require the curvature
itself to be pointwise bounded by some polynomial function.  The type I blowup
limits constructed in \cite{Weinkove} have bounded curvature and so
automatically satisfy this hypothesis.

\begin{defn} A connection $\nab$ on $\mathbb R^n$ has \textit{polynomial energy
growth} about $y \in \bRn$ if there exists a polynomial $p$ such that
\[ \int_{B_{y}(r)}{ |F_{\N}|^2 dV} \leq p(r). \]
\end{defn}

\begin{defn} A connection $\nab$ on $\mathbb R^n$ has \textit{polynomial
curvature growth} if there exists a polynomial $p$ such that for all $x \in
\mathbb R^n$, one has $\brs{F_{\N}}(x) \leq p(r(x))$.
\end{defn}

\begin{lemma}\label{Fballbd}
Let $\nab_t \in \mathcal{A}_E \times [0,T)$ be some solution to Yang-Mills flow
on $(M^n, g)$, with $n \geq 4$.  Given $t_1 \in [0,T)$, there exists $R > 0$ and
$C = C(t_1, g)$ such that for all $x_0 \in M$, $t \in [t_1,T)$ and $r \leq R$ we
have
\begin{equation*}
\int_{B_{x_0}(r)}{|F_{\nab_t}|^2  dV_g} \leq C e^{1/4} \mathcal{YM}(\nab_0)
r^{n-4} t_1^{\frac{4-n}{2}}.
\end{equation*}

\begin{proof}
Let $\rho(x,y)$ denote the distance function on $M$ between $x$ and $y$, and let
$G^M_0$ denote the heat kernel of the manifold $M$ with respect to the metric
$g$ based at the center point $x_0$ at time $t_0$.  First, using Proposition
\ref{monform} and then appealing to a Euclidean-type heat kernel upper bound,
see for instance (\cite{Li} Theorem 13.4) we obtain, for any $t < t_0$, for some
$C_1$ dependent on $T$ and $M$ coming from Theorem \ref{thm:hammonoform} (which
introduces the $C_M$),
\begin{align}
\begin{split}\label{eq:upprbd}
(t_0 - t)^2 \int_{B_{x_0}(r)} |F_{t}|^2 G^M_{0}(x,t) dV_g 
&\leq (t_0 - t)^2 \int_{M} |F_{t}|^2 G^M_{0}(x,t) dV_g\\
&\leq C_M t_0^2 \int_{M} |F_{0}|^2 G^M_{0}(x,0) dV_g + C_M t_0^2
\mathcal{YM}(\N_0)\\
&\leq C_1 t_0^{\frac{4-n}{2}} \mathcal{YM}(\nab_0).
\end{split}
\end{align}
Also, appealing to a local Euclidean-type heat kernel lower bound, (\cite{Li}
Theorem 13.8) we have, for sufficiently small $R$ and all $r \leq R$,
\begin{align}
\begin{split}\label{eq:lwrbd}
(t-t_0)^2 \int_{B_{x_0}(r)}{|F_{\nab}|^2 G_{0}^M dV_g } & \geq  C_2 (t -t_0)^2
\int_{B_{x_0}(r)}{|F_t|^2 (4 \pi (t_0 - t))^{-\frac{n}{2}}
e^{-\frac{\rho(x,x_0)^2}{4(t_0-t)}} dV_g}\\
&= C_2 (t -t_0)^{\frac{4 -n }{2}} e^{- \frac{r^2}{4(t_0-t)}} \int_{B_{x_0}(r)}{
|F_{t}|^2 dV_g }.
\end{split}
\end{align}
By combining inequalities \eqref{eq:upprbd} and \eqref{eq:lwrbd} and setting $C
:= \tfrac{C_1}{C_2}$ we have
\begin{align}
\begin{split}\label{eq:PEGeq1}
\int_{B_{x_0}(r)}{ |F_{t}|^2 dV_g } & \leq C (t -t_0)^{\frac{n - 4}{2}}
(t_0)^{\frac{4-n}{2}} e^{\frac{r^2}{4(t_0-t)}} \mathcal{YM} (\nab_0) \\
& =C e^{\frac{r^2}{4(t_0-t)}} \left( 1- \frac{t}{t_0} \right)^{\frac{n-4}{2}}
\mathcal{YM}(\nab_0).
\end{split}
\end{align}
Now take $ t_0 = t + r^2$ and observe that 
\begin{align}\label{eq:PEGeq2}
\begin{split}
\left( 1- \frac{t}{t + r^2} \right)^{\frac{n-4}{2}} &= \left( \frac{r^2}{t +
r^2} \right)^{\frac{n-4}{2}}\\
& = \left( {r^2}\right)^{\frac{n-4}{2}} \left( {t + r^2} \right)^{ -
\frac{n-4}{2}}\\
& \leq r^{n-4} t^{ - \frac{n-4}{2}}\\
& \leq r^{n-4}  (t_1)^{ \frac{4-n}{2}}.
\end{split}
\end{align}
Applying \eqref{eq:PEGeq2} in \eqref{eq:PEGeq1} we conclude
\begin{equation}
\int_{B_{x_0}(r)}{ |F_{t}|^2 dV_g } \leq C e^{1/4}  \mathcal{YM}(\nab_0) r^{n-4}
t_1^{\frac{4-n}{2}}.
\end{equation}
The result follows.
\end{proof}
\end{lemma}

\begin{prop}\label{pegprop} Let $\nab_t$ be a solution to Yang-Mills flow on
$(M, g)$ which exists for $t \in [0,T)$.  Fix some local framing and let $\gG$
denote the coefficient matrix of $\nab$, and let $\nab^{r_i}_s(y)$ be the
connection with coefficient matrix $\gG_s^{r_i}(y) := r_i \gG(\exp_{x_0}(r_i y),
T + r_i^2 s)$.  Assume that $\nab^{r_i}_s(y)$ converges strongly on $M \times
(-\infty,0)$ as $r_i \to 0$ to a self-similar solution $\nab^{\infty}_s$. Then
for any $r>0$ we have
\begin{equation}
\int_{B_0(r)}{| F_{\nab^{\infty}}(y,-1)|^2 dy} \leq e^{1/4} \mathcal{YM} \left(
\frac{T}{2} \right)^{- \frac{n-4}{2}} r^{n-4}.
\end{equation}

\begin{proof}
For each $i \in \mathbb{N}$ the following equality holds
\begin{equation}
\int_{B_0(r)}{| F_{\nab^{r_i}}(y,s)|^2 dV_y} = \int_{B_{x_0}(r r_i)}{r_i^4 |
F_{\nab}( \exp_{x_0}(r_i x),{T + r_i^2 s})|^2 r_i^{-n} dV_x}.
\end{equation}
We look at the temporal slice $s=-1$ and choosing $i$ sufficiently large to
ensure that $r_i^2 \leq \frac{T}{2}$. Then applying Lemma \ref{Fballbd} we have
\begin{equation}
\int_{B_0(r)}{| F_{\nab^{r_i}}(y,-1)|^2 dV_y} \leq 2 e^{1/4}
\mathcal{YM}(\nab_0) \left( \frac{T}{2} \right)^{-\frac{n-4}{2}} r^{n-4},
\end{equation}
sending $i \to \infty$ yields the result.
\end{proof}
\end{prop}

\section{Variational properties} \label{vars}

In this section we establish some fundamental variational properties of the
Yang-Mills entropy.  We begin by establishing first and second variation
formulas for the entropy functional $\FF$, including variations of the point in
spacetime.  We begin with some preliminary integration by parts formulas, then
use these to obtain the first and second variations.  These yield as corollaries
that solitons are characterized as critical points for the $\FF$-functional.  We
combine these calculations in \S \ref{ss:Ffunctnentropy} to establish that the
entropy is indeed achieved for a soliton with polynomial energy growth, realized
by the $\FF$-functional based at the basepoint of the given soliton.  Moreover
this point uniquely realizes the entropy, unless the soliton exhibits some flat
directions.

\subsection{Preliminary calculations}\label{s:prelimcalcs}

\begin{lemma} Let $\nab^{\infty}$ satisfy \eqref{eq:solitondef} with polynomial
energy growth. Then setting $\nab(x):= \frac{1}{\sqrt{t_0}}\nab^{\infty}\left(
\frac{x-x_0}{\sqrt{t_0}} \right)$, it holds that for all $\theta \geq 0$,
\begin{align}
\begin{split}\label{peg}
I_{\theta}(\N) := \int_{\mathbb{R}^n}{ |x-x_0|^{\theta} | F_{\nab} |^2
e^{\frac{-|x-x_0|^2}{4t_0}} dV} < \infty. 
\end{split}
\end{align} 

\begin{proof}
We observe that $\nab$ as defined above blows up at $(x_0,t_0)$, and, via change
of variables, satisfies
\begin{align*}
\int_{B_{x_0}(r)}{| F_{\N} |^2 dV} &=
\int_{B_{x_0}\left(\frac{r}{\sqrt{t_0}}\right)}{| F_{\nab^{\infty}}|^2
t_0^{\frac{n-2}{2}}dV} \\
& \leq p \left( \frac{r}{\sqrt{t_0}}\right) t_0^{\frac{n-2}{2}}.
\end{align*}
For each $r \in \mathbb{R}$, set $A_{x_0}(r) := B_{x_0}(r) / B_{x_0}(r-1)$. Then
partitioning $\mathbb{R}^n$ into a union of annuli yields
\begin{align*}
I_{\theta} &= \sum_{r = 1}^{\infty} \int_{A_{x_0}(r)}{|x-x_0|^{\theta} |F|^2
e^{\frac{-|x-x_0|^2}{4t_0}} dV} \\
& \leq \sum_{r = 1}^{\infty} r^{\theta} e^{\frac{(r-1)^2}{4 t_0}} 
\int_{A_{x_0}(r)}{ |F|^2 dV}\\
& \leq \sum_{r = 1}^{\infty} \int_{B_{x_0}(r)}{|F|^2 dV}.
\end{align*}
Incorporating the assumption of polynomial energy growth \eqref{peg}, we
conclude that
\begin{equation}
I_{\theta} \leq \sum_{r=1}^{\infty} r^{\theta} e^{\frac{-(k-1)^2}{4 t_0}}
p\left( \frac{r}{\sqrt{t_0}} \right) t_0^{\frac{n-2}{2}} < \infty.
\end{equation}
The result follows.
\end{proof}
\end{lemma}

\begin{lemma}\label{lem:derividgen} Let $\N$ be a $\left(\chi,\tau
\right)$-soliton with polynomial energy growth and let $G_0 = e^{- \frac{\brs{x
- x_0}^2}{4 t_0}}$.  Then for any vector fields $\xi = \xi^i \del_i$ such that
$|\xi|^2 G_0 \in L^{\infty}(\bRn)$,
\begin{align}
\begin{split}\label{eq:gen5plemid}
\int_{\bRn}{\xi^i (x-x_0)^i \left| F \right|^2 G_0 dV} &= 8 t_0 \int_{\bRn}F_{pu
\ga}^{\gb} F_{iu\gb}^{\ga} \left( \del_p \xi^i + \frac{1}{2}\left(\frac{x_0}{
t_0} - \frac{\chi}{ \tau}  + x \left(\frac{1}{\tau} - \frac{1}{t_0} \right)
\right)^p \xi^i \right)G_0 dV\\
& \hsp + 2 t_0 \int_{\bRn}(\del_i \xi^i) \left| F \right|^2 G_0 dV.
\end{split}
\end{align}
\begin{proof}
We let $\xi = \xi^{i}\del_i$ be a smooth vector field on $M=\mathbb{R}^n$ and
$\eta \in C_c^{\infty}(\mathbb{R}^n)$ with $|\eta| \leq 1$. Observe that
\begin{equation}
\frac{\del G_0}{\del x^i} = \frac{- (x-x_0)^i}{2t_0} G_0.
\end{equation}
Applying this equality and integrating by parts we obtain
\begin{align}
\begin{split}\label{5ptlemmaeq}
\int_{\bRn}{\xi^i(x-x_0)^i |F|^2 \eta G_0 dV} &= -2t_0 \int_{\bRn}{\xi^i |F|^2
(\del_i G_0) \eta dV}\\
&= 2 t_0 \int_{\bRn}{ (\del_i \left( \xi^i |F|^2  \eta \right)) G_0 dV}\\
&= 2  t_0 \int_{\bRn}{ \left( (\del_i \xi^i \eta) |F|^2 + \eta\xi^i (\del_i
|F|^2) \right) G_0 dV}.
\end{split}
\end{align}
Additionally by an application of the Bianchi identity we have
\begin{align*}
\int_{\bRn}{ \langle D^*F , \xi \hook F \rangle \eta G_0 d V}
&= \int_{\bRn}{ \left(( \N_p F_{pu \ga}^{\gb}) F_{i u \gb}^{\ga} \right) \xi^i
\eta G_0 d V}\\
&=  \int_{\bRn}{\left( \N_p \left( F_{pu \ga}^{\gb} F_{i u \gb}^{\ga} \right) -
F_{pu \ga}^{\gb}( \N_p F_{i u \gb}^{\ga})  \right) \xi^i \eta G_0 d V}\\
&=  -\int_{\bRn}{\left( F_{pu \ga}^{\gb} F_{i u \gb}^{\ga} \right) \N_p (\xi^i
\eta G_0) d V} + \int_{\bRn}{ F_{pu \ga}^{\gb} (\N_u F_{p i \gb}^{\ga} + \N_i
F_{ u p \gb}^{\ga})   \xi^i \eta G_0 d V}\\
&=  -\int_{\bRn}{\left( \left( F_{pu \ga}^{\gb} F_{i u \gb}^{\ga} \right) \left(
(\del_p(\eta \xi^i)) - \frac{(x-x_0)^p}{2 t_0} \eta \xi^i \right) \right)G_0 d
V}\\
&\hsp + \int_{\bRn}{\left(  F_{pu \ga}^{\gb} (\N_u F_{p i \gb}^{\ga}\right)
\xi^i \eta G_0 d V} - \frac{1}{2}\int_{\bRn}{\left(  \N_i (F_{pu \ga}^{\gb} F_{
pu \gb}^{\ga})  \right) \xi^i \eta G_0 d V}\\
&=  - \int_{\bRn}{\left( \left( F_{pu \ga}^{\gb} F_{i u \gb}^{\ga} \right)
\left( (\del_p(\eta \xi^i)) - \frac{(x-x_0)^p}{2 t_0} \eta \xi^i \right)
\right)G_0 d V}\\
&\hsp + \int_{\bRn}{\left(  F_{pu \ga}^{\gb} (\N_u F_{p i \gb}^{\ga} )\right)
\xi^i \eta G_0 d V} + \frac{1}{2}\int_{\bRn}{\left(  \del_i (|F|^2)  \right)
\xi^i \eta G_0 d V}.
\end{align*}
We multiply through the equality by $4t_0$ and isolate the last term on the
right.
 Applying this to \eqref{5ptlemmaeq},
\begin{align*}
\int_{\bRn}{ \xi^i(x-x_0)^i |F|^2 \eta G_0 dV} &=  4 t_0 \int_{\bRn}{ F_{pu
\ga}^{\gb} (F_{i u \gb}^{\ga}) \left( (\del_p(\eta \xi^i)) - \frac{(x-x_0)^p}{2
t_0} \eta \xi^i \right)G_0 dV} \\
& \hsp - 4 t_0 \int_{\bRn}{\left(  F_{pu \ga}^{\gb} (\N_u F_{p i \gb}^{\ga}) + 
(D^*F)_{u \ga}^{\gb} F_{i u \gb}^{\ga}\right) \xi^i \eta G_0 d V}\\
& \hsp + 2 t_0 \int_{\bRn}{ (\N_i (\xi^i \eta)) |F|^2 G_0 dV}.
\end{align*}
We let $\eta_R$ be a cut off function with support within $B(R)$ which cuts off
to zero linearly between $B(R)$ and $B(R+1)$. Applying this to the above
expression and sending $R \to \infty$, it follows from the Dominated Convergence
Theorem (one will see that for each of the test functions $\eta$ we insert this
holds) that we obtain 
\begin{align}
\begin{split}\label{eq:zhanggen5pc0}
\int_{\bRn}{ \xi^i(x-x_0)^i |F|^2 G_0 dV}
&= 2 t_0 \int_{\bRn}{ (\del_i (\xi^i)) |F|^2 G_0 dV} + 4 t_0 \int_{\bRn}{ F_{pu
\ga}^{\gb} (F_{i u \gb}^{\ga}) \left( (\del_p( \xi^i)) - \frac{(x-x_0)^p}{2 t_0}
\xi^i \right)G_0 dV} \\
& \hsp - 4 t_0 \int_{\bRn}{\left(  F_{pu \ga}^{\gb} (\N_u F_{p i \gb}^{\ga} ) + 
(D^*F)_{u \ga}^{\gb} F_{i u \gb}^{\ga}\right) \xi^i G_0 d V}.
\end{split}
\end{align}
Then we manipulate the latter term of above. Since the integrand consists of an
inner product against a skew quantity we may consider the skew projection of $\N
F$ onto proper components. A subsequent application of the Bianchi identity and
inclusion of divergence term yields
\begin{align*}
\int_{\bRn}{\left(  F_{pu \ga}^{\gb} (\N_u F_{p i \gb}^{\ga} )\right) \xi^i G_0
d V} &= \frac{1}{2}\int_{\bRn}{\left(  F_{pu \ga}^{\gb} (\N_u F_{p i \gb}^{\ga}
- \N_p F_{u i \gb}^{\ga} )\right) \xi^i G_0 d V} \\
&= - \frac{1}{2}\int_{\bRn}{\left(  F_{pu \ga}^{\gb} (\N_i F_{up \gb}^{\ga}
)\right) \xi^i G_0 d V} \\
&= - \frac{1}{4}\int_{\bRn}{ \N_i \lbr |F|^2 \rbr \xi^i  G_0 d V} \\
&= \frac{1}{4}\int_{\bRn}{  |F|^2 \N_i \lbr \xi^i  G_0 \rbr d V} \\
&=  \frac{1}{4}\int_{\bRn}{  |F|^2 ( \del_i \xi^i ) G_0 d V} -
\frac{1}{8t_0}\int_{\bRn}{  |F|^2 \xi^i (x-x_0)^i  G_0 d V}.
\end{align*}
We insert this identity into \eqref{eq:zhanggen5pc0} and obtain
\begin{align*}
\begin{split}
\int_{\bRn}{ \xi^i(x-x_0)^i |F|^2 G_0 dV}
&= t_0 \int_{\bRn}{ (\del_i (\xi^i)) |F|^2 G_0 dV} + 4 t_0 \int_{\bRn}{ F_{pu
\ga}^{\gb} (F_{i u \gb}^{\ga}) \left( (\del_p( \xi^i)) - \frac{(x-x_0)^p}{2 t_0}
\xi^i \right)G_0 dV} \\
& \hsp + \frac{1}{2}\int_{\bRn}{ |F|^2 (x-x_0)^i \xi^i G_0 dV} - 4 t_0
\int_{\bRn}{(D^*F)_{u \ga}^{\gb} F_{i u \gb}^{\ga} \xi^i G_0 d V}.
\end{split}
\end{align*}
We note that since $\N$ is a self-similar solution based at
$(\chi,\tau)$ we replace $D^*F = \tfrac{\chi - x}{2 \tau} \hook F_{\N}$ to yield
\begin{align*}
\begin{split}
\int_{\bRn}{\left| F \right|^2 (x-x_0)^i \xi^i G_0 dV} &=8 t_0 \int_{\bRn}F_{pu
\ga}^{\gb} F_{iu\gb}^{ga} \left( \del_p \xi^i + \frac{1}{2}\left(\frac{x_0}{
t_0}- \frac{\chi}{\tau}  -x \left(\frac{1}{t_0}-\frac{1}{\tau} \right) \right)^p
\xi^i \right)G_0 dV \\
& \hsp + 2 t_0 \int_{\bRn}(\del_i \xi^i ) \left| F \right|^2 G_0 dV,
\end{split}
\end{align*}
as claimed.
\end{proof}
\end{lemma}

\begin{cor}[Soliton Identities]\label{5ptlemSolId} Let $\N \in \mathfrak{S}$
satisfy polynomial energy growth and set $G_0:  = e^{-\frac{\brs{x - x_0}^2}{4
t_0}}$. Let $V = V^i \del_i$ be any constant vector field on $\mathbb{R}^n$.
Then the following equalities hold.
\begin{enumerate}
\item[(a)] $\int_{\bRn} \left( \frac{|x-x_0|^2}{4} + \frac{(4 - n)}{2} 
\right)\brs{F}^2 G_0 dV 
= - \int_{\bRn}\langle  \left( x\left(t_0 - 1 \right) + x_0 \right) \hook F,
(x-x_0) \hook F \rangle G_0 dV$,
\item[(b)] $\int_{\bRn}{\frac{\langle x-x_0, V \rangle}{2} \left| F \right|^2
G_0 dV} =  - 2\int_{\bRn}{ \langle \left(x \left( t_0 - 1 \right) + x_0 \right)
\hook F, V \hook F \rangle G_0 dV}$.
\end{enumerate}

\begin{proof}
We set $\chi= 0$ and $\tau =1$ to incorporate the soliton equation
\eqref{eq:solitondef}. To obtain (a) we insert $\xi^i = \tfrac{(x-x_0)^i}{4}$
into \eqref{eq:gen5plemid}, and for (b) we use $\xi^i = \frac{V^i}{2}$.
\end{proof}
\end{cor}

\begin{cor} \label{5partlemma} Let $\N \in \mathfrak{S}_{x_0,t_0}$ be a
$(x_0,t_0)$-soliton
with polynomial energy growth.  Let $V=V^i \del_i$ be any constant vector field
on $\mathbb{R}^n$ and $\gamma \in [1,n] \cap \mathbb{N}$. Then the following
equalities hold.
\begin{enumerate}
\item[(a)] $
\int_{\mathbb{R}^n}{\left( \left(4- n \right) + \frac{| x - x_0 |^2}{2t_0} 
\right)|F|^2 G_0 dV} = 0$,
\item[(b)] $\int_{\bRn}{(x-x_0)^{\gamma} |F|^2 G_0 dV} = 0$,
\item[(c)]  $\int_{\bRn}{|x-x_0|^4 |F|^2 G_0 dV} = 4(n-2)(n-4)t_0^2 \int_{\bRn}{
|F|^2 G_0 dV}  - 64 t_0^3 \int_{\bRn}{|D^*F|^2 G_0 dV}$,
\item[(d)] $\int_{\bRn}{|x-x_0|^2 \langle V , x-x_0 \rangle |F|^2 G_0 dV} = 
\int_{\bRn}{\langle (V \hook F), D^*F \rangle G_0 dV} = 0$,

\item[(e)] $\int_{\bRn}{ \langle V, x-x_0 \rangle^2 |F|^2 G_0 dV} = 2t_0
\int_{\bRn}{ |V|^2 |F|^2 G_0 dV}  - 8 t_0 \int_{\bRn}{|V \hook F|^2G_0 dV}$.
\end{enumerate}

\begin{proof}
Starting with \eqref{eq:gen5plemid}, we set $\tau= t_0$ and $\chi = x_0$.
This yields
\begin{align}
\begin{split}\label{eq:zhang5pc}
\int_{\bRn} |F|^2 (x-x_0)^i \xi^i G_0 dV &=8 t_0 \int_{\bRn}F_{pu \ga}^{\gb}
F_{iu\gb}^{\ga} \left( \del_p  \xi^i\right)G_0 dV + 2 t_0 \int_{\bRn}(\del_i
\xi^i ) \left| F \right|^2 G_0 dV.
\end{split}
\end{align}
We approach the listed quantities of the lemma with this identity.

\begin{enumerate}

\item[(a)] This immediately follows by setting $\xi^i := \frac{(x-x_0)^i}{4
t_0}$.

\item[(b)] This immediately follows by setting $\xi^i := \delta^{\gamma i}$.

\item[(c)] Set $\xi^i := |x-x_0|^2 (x-x_0)^i$. Prior to solving we compute the
following derivative:
\begin{align*}
\del_p (|x-x_0|^2 (x-x_0)^i) &= 2 (x-x_0)^p (x-x_0)^i +  |x-x_0|^2 \delta_{ip}.
\end{align*}
Applying this to \eqref{eq:zhang5pc} gives
\begin{align*}
\int_{\bRn}{|x-x_0|^4 |F|^2 G_0 dV} &= 2t_0 \int_{\bRn}{ (2+n) |x-x_0|^2 |F|^2
G_0 dV} \\
& \hsp + 8t_0 \int_{\bRn}{F_{pu \ga}^{\gb} F_{i u \gb}^{\ga} \left( 2 (x-x_0)^p
(x-x_0)^i +  |x-x_0|^2 \delta^{ip} \right)G_0 dV} \\
&= 2(n-2)t_0 \int_{\bRn}{  |x-x_0|^2 |F|^2 G_0 dV} - 16 t_0 \int_{\bRn}{ | F
\hook (x-x_0) |^2G_0 dV}.
\end{align*}
Now we replace the first term with the identity of (a) and the second term with
the $(x_0,t_0)$-soliton equation \eqref{eq:X0T0solitondef} to conclude that
\begin{align*}
\int_{\bRn}{|x-x_0|^4 |F|^2 G_0 dV} &= 4(n-2)(n-4)t_0^2 \int_{\bRn}{ |F|^2 G_0
dV}  - 64 t_0^3 \int_{\bRn}{|D^*F|^2 G_0 dV}.
\end{align*}
\item[(d)] To prove this identity it will require applying two different test
functions to \eqref{eq:zhang5pc}. First we set $\xi^i := |x-x_0|^2 V^i$. Then
using (b) and the $(x_0,t_0)$-soliton equation \eqref{eq:X0T0solitondef} we
obtain
\begin{align*}
\int_{\bRn}{ |x-x_0|^2 \langle V ,(x-x_0) \rangle |F|^2 G_0 dV} &= 4 t_0
\int_{\bRn}{ \langle (x-x_0), V \rangle |F|^2 G_0 dV} \\
& \hsp + 16 t_0 \int_{\bRn}{F_{pu \ga}^{\gb} F_{i u \gb}^{\ga}  (x-x_0)^p V^i
G_0 dV} \\
& = - 32 t_0^2 \int_{\bRn}{(D^*F)_{u \ga}^{\gb} F_{i u \gb}^{\ga}  V^i G_0 dV}. 
\end{align*}
Now we will instead consider $\xi^i := \langle V, (x-x_0) \rangle (x-x_0)^i$.
Prior to this we differentiate
\begin{align*}
\del_p \left[\langle V, (x-x_0) \rangle (x-x_0)^i \right] &= V^p (x-x_0)^i +
\langle V, (x-x_0) \rangle \delta_{ip}.
\end{align*}
therefore we have, applying (b) to \eqref{eq:zhang5pc},
\begin{align*}
\int_{\bRn}{ \langle V, (x-x_0) \rangle |x-x_0 |^2 |F|^2 G_0 dV} &= 2t_0
\int_{\bRn}{ ( V^i (x-x_0)^i + n \langle V, (x-x_0) \rangle) |F|^2 G_0 dV} \\
& \hsp + 8 t_0 \int_{\bRn}{F_{pu \ga}^{\gb} F_{i u \gb}^{\ga} \left( V^p
(x-x_0)^i + \langle V, (x-x_0) \rangle \delta^{ip}  \right)G_0 dV}\\
&= 8 t_0 \int_{\bRn}{F_{pu \ga}^{\gb} F_{i u \gb}^{\ga} \left( V^p (x-x_0)^i
\right)G_0 dV}\\
&= - 16t_0^2 \int_{\bRn}{(D^*F)_{u \ga}^{\gb} F_{pu \gb}^{\ga}V^p G_0 dV}.
\end{align*}
By equality of the two expressions we conclude that
\begin{equation*}
\int_{\bRn}{ \langle V, (x-x_0) \rangle |x-x_0 |^2 |F|^2 G_0 dV} =
\int_{\bRn}{\langle D^*F, F \hook V \rangle G_0 dV}  = 0.
\end{equation*}
\item[(e)] Set $\xi^i := \langle V, x-x_0 \rangle V^i$. This quantity
differentiated is precisely
\begin{align*}
\int_{\bRn}{ \langle V, x-x_0 \rangle^2 |F|^2 G_0 dV} &= 2t_0 \int_{\bRn}{ |V|^2
|F|^2 G_0 dV}  - 8 t_0 \int_{\bRn}{|F \hook V|^2G_0 dV}.
\end{align*}
\end{enumerate}
The final result follows, and the proof is complete.
\end{proof}
\end{cor}

One consequence of these identities is that the Yang-Mills flow in dimension
$n=4$ cannot exhibit type I singularities.

\begin{prop}\label{prop:YMsoliton4flat} Let $E \to (M^4, g)$ be a smooth vector
bundle,
and suppose $\N_t$ is a solution to Yang-Mills flow on $E$ which exists on a
maximal time interval $[0,T)$, with $T < \infty$.  Then
\begin{align*}
\lim_{t \to T} (T - t) \brs{F_{\N_t}} = \infty.
\end{align*}
Moreover, any soliton on $\mathbb R^n$, $n \leq 4$, is flat.

\begin{proof}
Suppose to the contrary that there exists a $C \in \mathbb{R}$ so that
\begin{align*}
\lim_{t \to T} (T - t) \brs{F_{\N_t}} \leq C.
\end{align*}
By Weinkove's main theorem (pp.2 \cite{Weinkove}), we can construct a type I
blowup limit $\nab_{\infty}$ which is a self-similar solution, whose time $t=-1$
slice is a nonflat $(0,1)$-soliton.  By Corollary \ref{5partlemma} part (a)
since $\dim M = 4$, we conclude
\begin{equation}
\int_{\mathbb{R}^n}{ \frac{|x|^2}{2} |F_{\N_{\infty}}|^2 G_{0} dV} = 0.
\end{equation}
Thus it follows that $\nab_{\infty}$ is flat, but this is a contradiction since
by construction $\nab_{\infty}$ is nonflat. The result follows.
\end{proof}
\end{prop}

\subsection{First variation}\label{ss:firstvariation}

In this subsection we compute the first variation of the $\FF$-functional.  For
both the first and second variation computations we the dependence on $s$ will
be dropped for all terms except the varying base point $(x_s,t_s)$ during
coordinate computations.  Moreover, the variational calculations require the
variation of the connection to be in a certain weighted Sobolev space (see
\ref{Wdef}) which we suppress here.

\begin{prop}[First variation]\label{zYMfirstvar} Let $\gG_s$, $x_s$ and $t_s$ be
one parameter families of connections, points in $\mathbb R^n$, and positive
real numbers respectively, and set
\begin{equation}\label{defGs}
G_s(x) := (4 \pi t_s)^{\frac{-n}{2}} e^{\frac{-|x-x_s|^2}{4t_s} },
\end{equation}
and furthermore set
\begin{equation*}
\dot{t}_s := \frac{d t_s}{d s}, \hspace{0.1cm} \dot{x}_s := \frac{d x_s}{d s},
\hspace{0.1cm} \dot{\gG}_s := \frac{d \gG_s}{d s}.
\end{equation*}
Then
\begin{align}
\begin{split}\label{eq:zYMfirstvar}
\frac{d}{d s} \left[  \mathcal{F}_{x_s,t_s}(\nab_s)  \right] &=\dot{t}_s
\int_{\bRn}{\left( t_s \left(\frac{4- n}{2}\right) + \frac{| x - x_s |^2}{4} 
\right) |F_s|^2 G_s dV}\\
&\ \hsp +  t_s \int_{\bRn}{ \frac{\langle \dot{x}_s , x - x_s \rangle}{2}
|F_s|^2 G_s dV}\\
& \hsp +  4 t_s^2 \int_{\bRn}{ \left\langle \dot{\gG}_s , D_s^* F_s + \left(
\frac{(x-x_s)}{2 t_s} \hook F_s \right) \right\rangle G_s dV}.
\end{split}
\end{align}

\begin{proof}
We first differentiate the following expression
\begin{align*}
\frac{d}{d s} \left[ \int_{\bRn}{|F_s|^2 G_s dV} \right] &= \int_{\bRn}{ \left(
\frac{\del}{\del s} |F_s|^2 \right) G_s  dV} + \int_{\bRn}{|F_s|^2 \left(
\frac{\del}{\del s}G_s \right)dV}.
\end{align*}
We compute the first quantity on the right side of the equality:
\begin{align}
\begin{split}\label{eq:zYMfirstvar1}
\int_{\bRn}{ \left(\frac{\del}{\del s}|F_s|^2 \right) G_s dV}
&=  -2\int_{\bRn}{ \left( \nab_i \dot{\gG}_{j \ga}^{\gb} - \nab_j \dot{\gG}_{i
\ga}^{\gb} \right) F_{ij \gb}^{\ga} G dV} \\
&= 2\int_{\bRn}{ \dot{\gG}_{j \ga}^{\gb}( \nab_i  F_{ij \gb}^{\ga}) G dV} +
2\int_{\bRn}{ \dot{\gG}_{j \ga}^{\gb}  F_{ij \gb}^{\ga} (\nab_i G ) dV} \\
& \hsp - 2\int_{\bRn}{ \dot{\gG}_{i \ga}^{\gb} (\nab_j  F_{ij \gb}^{\ga}) G dV}
- 2\int_{\bRn}{\dot{\gG}_{i \ga}^{\gb}  F_{ij \gb}^{\ga} (\nab_j  G ) dV}\\
&= 4 \int_{\bRn}{ \langle  \dot{\gG}_s, D_s^* F_s \rangle G_s dV} +
4\int_{\bRn}{\dot{\gG}_{j \ga}^{\gb}  F_{ij \gb}^{\ga} (\nab_i G ) dV}.
\end{split}
\end{align}
We differentiate $G_s$ and obtain
\begin{align}
\begin{split}\label{eq:sderivG}
\frac{\del}{\del s} \left[ G_s \right] = \left( \frac{-n}{2}
\frac{\dot{t}_s}{t_s} + \frac{\dot{t}_s |x - x_s|^2}{4 t_s^2} + \frac{\langle
\dot{x}_s, x - x_s \rangle}{2 t_s} \right) G_s.
\end{split}
\end{align}
Furthermore,
\begin{equation}\label{eq:spcderivG}
\nab_i G = - \frac{(x- x_s)^i}{2t_s}G.\end{equation}
Applying this to \eqref{eq:zYMfirstvar1} gives that
\begin{align*}
\int_{\bRn}{\dot{\gG}_{j \ga}^{\gb}  F_{ij \gb}^{\ga} (\nab_i G) dV} & =
\int_{\bRn}{\dot{\gG}_{j \ga}^{\gb}  F_{ij \gb}^{\ga} \left(- \frac{(x-
x_s)^i}{2t_s}  \right) G  dV}\\
& = \int_{\bRn}{\left\langle \dot{\gG}_s , \left( \frac{(x- x_s)}{2t_s}\hook 
F_s \right) \right\rangle G_s  dV}.
\end{align*}
So we conclude that
\begin{align}
\begin{split}\label{eq:zYMfirstvar2}
\frac{d}{d s} \left[ \int_{\bRn}{|F_s|^2 G_s dV} \right]  &= \int_{\bRn}{| F_s
|^2 \left( \frac{-n}{2} \frac{\dot{t}_s}{t_s} + \frac{\dot{t}_s |x - x_s|^2}{4
t_s^2} + \frac{\langle \dot{x}_s, x - x_s \rangle}{2 t_s} \right) G_s dV} \\
& \hsp + 4 \int_{\bRn}{\left\langle \dot{\gG}_s, D^* F_s + \left( \frac{(x-
x_s)}{2t_s} \hook F_s \right) \right\rangle G_s  dV}.
\end{split}
\end{align}
With this in mind we differentiate the expression
\begin{align*}
\frac{d}{d s} \left[  t_s^2 \int_{\bRn}{|F_s|^2 G_s dV} \right]
& = 2  t_s \dot{t}_s  \int_{\bRn}{|F_s|^2 G_s dV}  +  t_s^2
\left(\frac{\del}{\del s} \left[ \int_{\bRn}{|F_s|^2 G_s dV} \right]\right),
\end{align*}
and then applying \eqref{eq:zYMfirstvar2} we obtain that
\begin{align*}
\frac{d}{d s} \left[ t_s^2 \int_{\bRn}{|F_s|^2 G_s dV} \right] &= t_s^2
\int_{\bRn}{| F_s |^2 \left( \frac{-n}{2} \frac{\dot{t}_s}{t_s} +
\frac{\dot{t}_s |x - x_s|^2}{4 t_s^2} + \frac{\langle \dot{x}_s, x - x_s
\rangle}{2 t_s} + \frac{2\dot{t}_s}{t_s} \right) G_s dV} \\
& \hsp  + 4 t_s^2 \int_{\bRn}{ \left\langle \dot{\gG}_s , D_s^* F_s + \left(
\frac{(x-x_s)}{2 t_s} \hook F_s \right) \right\rangle G_s dV}.
\end{align*}
Reordering terms yields the result.
\end{proof}
\end{prop}

\begin{cor} \label{Fcrit} The point $(\nab, x_0, t_0)$ is a critical point of
the $\mathcal{F}$-functional if and only if $\nab$ is an $(x_0,t_0)$-soliton.
\begin{proof} If  $(\nab, x_0, t_0)$ is a critical point, then all partial
derivatives with respect to $t$, $x$ and $\gG$ vanish. We note that if we vary
only the connection coefficient matrix $\gG$ then we have
\begin{equation*}
0 =  4 t_s^2 \int_{\bRn} \left\langle \dot{\gG}_s , D_s^* F_s +
\left(\frac{(x-x_s)}{2 t_s} \hook F_s \right) \right\rangle G_s dV.
\end{equation*}
In particular, this implies that $\nab$ is an $(x_0,t_0)$-soliton. Conversely,
given a point $(x_0,t_0)$ and that $\N$ is a soliton, then we apply identities
(a) and (b) of Corollary \ref{5partlemma} to the variation identity of
Proposition \ref{zYMfirstvar}. Each quantity vanishes, yielding that
$(\N,x_0,t_0)$ is a critical point and the result follows.
\end{proof}
\end{cor}

\begin{prop}\label{solitonmonotonicity} Let $\nab_t \in \mathcal{A}_E \times
[0,T)$ be a solution to Yang-Mills flow with polynomial energy growth.  Then
$\la(\nab_t)$ is non-increasing in $t$.

\begin{proof} Let $t_1, t_2 \in \mathbb{R}$ with $t_1 < t_2 < T$. Given
$\epsilon > 0$ there exists $(x_0,t_0)$ such that
\begin{equation}\label{eq:solitonmonotonicity0}
\mathcal{F}_{x_0,t_0}(\nab_{t_2}) \geq \la(\nab_{t_2}) - \epsilon.
\end{equation}
Thus it follows from Hamilton's monotonicity formula that for any $\delta \in
(0, t_2)$ we have
\begin{equation}\label{eq:solitonmonotonicity1}
\mathcal{F}_{x_0, t_0 + t_2}(\nab_{t_2}, t_2) \leq \mathcal{F}_{x_0, t_0 +
t_2}(\nab_{t_2 - \delta}, t_2 - \delta).
\end{equation}
But we observe that
\begin{align}
\begin{split}\label{eq:solitonmonotonicity2}
\mathcal{F}_{x_0, t_0 + t_2} (\nab_{t_2}, t_2) &= 
\int_{\bRn}{|F_{\nab_{t_2}}|^2 e^{- \frac{|x-x_0|^2}{4((t_0 + t_2) - t_2)}}dV}
\\
&= \int_{\bRn}{|F_{\nab_{t_2}}|^2 e^{- \frac{|x-x_0|^2}{4 t_0}}dV} \\
& = \mathcal{F}_{x_0, t_0} (\nab_{t_2}).
\end{split}
\end{align}
Thus combining \eqref{eq:solitonmonotonicity1} and
\eqref{eq:solitonmonotonicity2} we have
\begin{align}
\begin{split}\label{eq:solitonmonotonicity3}
\mathcal{F}_{x_0,t_0}(\nab_{t_2}) &\leq \mathcal{F}_{x_0, t_0 + t_2}(\nab_{t_2 -
\delta}, t_2 - \delta) \\
& = \int_{\mathbb{R}^n}{|F_{\nab_{t_2 - \delta}}|^2 e^{ -\frac{|x-x_0|^2}{t_0 +
\delta}} dV}\\
&=\mathcal{F}_{x_0,t_0 + \delta}(\nab_{t_2 - \delta}).
\end{split}
\end{align}
We set $\delta = t_2 - t_1$ and observe that, combining
\eqref{eq:solitonmonotonicity0} with \eqref{eq:solitonmonotonicity3},
\begin{equation*}
\lambda(\nab_{t_2} ) - \epsilon \leq \mathcal{F}_{x_0,t_0}(\nab_{t_2}) \leq
\mathcal{F}_{x_0,t_0 + t_2 - t_1}(\nab_{t_1}) \leq \lambda (\nab_{t_1}).
\end{equation*}
Since this holds for each $\epsilon > 0$ we conclude the desired monotonicity.
The result follows.
\end{proof}
\end{prop}

\subsection{Second variation}\label{ss:secondvariation}

For the computations of the following proposition refer to the note of \S
\ref{ss:firstvariation} regarding the convention on variation parameter
subscripts.  Again, these variational calculations require the variation of the
connection to be in a certain weighted Sobolev space which we suppress.

\begin{prop}[Second variation]\label{zhang2ndvar} Let $\gG_s$, $x_s$ and $t_s$
be one parameter families of connections, points in $\mathbb R^n$, and positive
real numbers respectively, and set
\begin{equation*}
\dot{t}_s := \frac{d t_s}{d s},\hspace{0.1cm} \ddot{t}_s := \frac{d \dot{t}_s}{d
s}, \hspace{0.1cm} \dot{x}_s := \frac{d x_s}{d s}, \hspace{0.1cm} \ddot{x}_s :=
\frac{d \dot{x}_s}{d s}, \hspace{0.1cm} \dot{\gG}_s := \frac{d \gG_s}{d s},
\hspace{0.1cm} \ddot{\gG}_s := \frac{d \dot{\gG}_s}{d s}.
\end{equation*}
Then
\begin{align*}
\begin{split}
\frac{d^2}{d s^2} & \left[ \mathcal{F}_{x_s,t_s} (\nab_s) \right] \\
& = \int_{\bRn}{\left(2(\ddot{t}_s t_s + \dot{t}_s^2)  + t_s^2 (\mathsf{g}_s^2 +
\dot{\mathsf{g}}_s) + 4 \dot{t}_s t_s \mathsf{g}_s \right)|F_s|^2 G_s dV} \\
& \hsp + \int_{\bRn}{\left\langle (8 \dot{t}_s t_s + 2t_s^2
\mathsf{g}_s)\dot{\gG}_s  + 8 t_s^2 \ddot{\gG}_s, S_{x_s,t_s}(\nab_s)
\right\rangle G_s  dV}\\
& \hsp + 4 t_s^2 \int_{\bRn}{ \left\langle \dot{\gG}_s, D_s^* D_s \dot{\gG}_s +
\left( \frac{x-x_s}{2t_s} \right) \hook D_s \dot{\gG}_s  + [\dot{\gG}_s,
F_s]^{\#} \right\rangle G_s dV}\\
& \hsp  -4 t_s \int_{\bRn}{\left\langle \dot{\gG}_s, \left( \dot{t}_s (x-x_s)
+\dot{x}_s \right) \hook F_s \right\rangle G_s dV}.
\end{split}
\end{align*}
where 
\begin{equation}
\mathsf{g}_s :=\left( - \frac{n \dot{t}_s}{2 t_s} + \frac{\dot{t}_s |x -
x_s|^2}{4 t_s^2} + \frac{\langle \dot{x}_s, x - x_s \rangle}{2 t_s} \right).
\end{equation}

\begin{proof}
Prior to computing the main expression we perform some necessary side
computations of differentiation. First observe that
\begin{equation*}
\frac{\del \mathsf{g}}{\del x^i} = \left( \frac{\dot{t}_s}{2 t_s} (x-x_s)^i +
\frac{\dot{x}_s^i}{2t_s} \right).
\end{equation*}
We next take the second variation of $\mathcal{F}_{x_s,t_s}$ and separate the
resulting expression into labeled quantities:
\begin{align}
\begin{split}\label{eq:4part2ndvar}
\frac{d^2}{d s^2} \left[ \int_{\bRn} | F_s|^2 G_s dV  \right]
&= \frac{d}{ds} \left( \int_{\bRn}{ \left( 2 \langle D_s \dot{\gG}_s , F_s
\rangle  + \mathsf{g}_s |F_s|^2 \right) G_s dV}  \right)\\
&=  \int_{\bRn}{ \left( \underset{T_1}{\underbrace{2 \langle \del_s (D_s
\dot{\gG}_s) , F_s \rangle}} + \underset{T_2}{\underbrace{2 \langle D_s
\dot{\gG}_s , D_s \dot{\gG}_s \rangle}} +
\underset{T_3}{\underbrace{(\dot{\mathsf{g}}_s + \mathsf{g}_s^2) |F_s|^2}} +
\underset{T_4}{\underbrace{4 \mathsf{g}_s \langle D \dot{\gG}_s, F_s \rangle }}
\right)G_s dV}.
\end{split}
\end{align}
We first address the $T_1$ quantity. Since
\begin{align*} 
\del_s D_s \dot{\gG}_s &= \del_s \left( \nab_i \dot{\gG}_{j \ga}^{\gb} \right)\\
&= \del_t \left( \del_i \dot{\gG}_{j \ga}^{\gb} - \gG_{i \ga}^{\gd} \dot{\gG}_{j
\gd}^{\gb} + \gG_{i \gd}^{\gb} \dot{\gG}_{j \ga}^{\gd} \right)\\
&=   \del_i \ddot{\gG}_{j \ga}^{\gb} - \dot{\gG}_{i \ga}^{\gd} \dot{\gG}_{j
\gd}^{\gb}  - \gG_{i \ga}^{\gd} \ddot{\gG}_{j \gd}^{\gb} + \dot{\gG}_{i
\gd}^{\gb} \dot{\gG}_{j \ga}^{\gd} + \gG_{i\gd}^{\gb} \ddot{\gG}_{j \ga}^{\gd},
\end{align*}
it follows that
\begin{align*}
\del_s \left( \nab_i \dot{\gG}_{j \ga}^{\gb} - \nab_j \dot{\gG}_{i \ga}^{\gb}
\right) &=  \del_i \ddot{\gG}_{j \ga}^{\gb} - \dot{\gG}_{i \ga}^{\gd}
\dot{\gG}_{j \gd}^{\gb}  - \gG_{i \ga}^{\gd} \ddot{\gG}_{j \gd}^{\gb} +
\dot{\gG}_{i \gd}^{\gb} \dot{\gG}_{j \ga}^{\gd} + \gG_{i \gd}^{\gb}
\ddot{\gG}_{j \ga}^{\gd} \\
& \hsp  - \del_j \ddot{\gG}_{i \ga}^{\gb} + \dot{\gG}_{j \ga}^{\gd} \dot{\gG}_{i
\gd}^{\gb}  + \gG_{j \ga}^{\gd} \ddot{\gG}_{i \gd}^{\gb} - \dot{\gG}_{j
\gd}^{\gb} \dot{\gG}_{i \ga}^{\gd} - \gG_{j \gd}^{\gb} \ddot{\gG}_{i
\ga}^{\gd}\\
&=  D_i \ddot{\gG}_{j \ga}^{\gb} - 2 \dot{\gG}_{i \ga}^{\gd} \dot{\gG}_{j
\gd}^{\gb}  + 2 \dot{\gG}_{i \gd}^{\gb} \dot{\gG}_{j \ga}^{\gd} .
\end{align*}
Now applying this to the expression above,
\begin{align*}
\int_{\bRn}{\langle \del_s (D_s \dot{\gG}_s), F_s \rangle G_s dV_g}  &=
-\int_{\bRn}{\left( \nab_i \ddot{\gG}_{j \ga}^{\gb} - \nab_j \ddot{\gG}_{i
\ga}^{\gb} \right) F_{ij \gb}^{\ga} G dV} - 2 \int_{\bRn}{\left( \dot{\gG}_{i
\ga}^{\gd} \dot{\gG}_{j \gd}^{\gb} - \dot{\gG}_{i \gd}^{\gb} \dot{\gG}_{j
\ga}^{\gd} \right) F_{ij \gb}^{\ga} G dV} \\
 &=  4 \int_{\bRn}{\ddot{\gG}_{j \ga}^{\gb}  \left(\nab_i  F_{ij \gb}^{\ga}
\right) G dV} + 4 \int_{\bRn}{ \left( \ddot{\gG}_{j \ga}^{\gb} \right) F_{ij
\gb}^{\ga} \left(\nab_i   G \right)dV}\\
 & \hsp- 2 \int_{\bRn}{ \left(  \dot{\gG}_{i \ga}^{\gd} \dot{\gG}_{j \gd}^{\gb}
F_{ij \gb}^{\ga} - \dot{\gG}_{i \gd}^{\gb} F_{ij \gb}^{\ga} \dot{\gG}_{j
\ga}^{\gd} \right) G dV}\\
  &=  4 \int_{\bRn}{\ddot{\gG}_{j \ga}^{\gb} \left(\nab_i  F_{ij \gb}^{\ga}
\right) G dV} - 4 \int_{\bRn}{\ddot{\gG}_{j \ga}^{\gb} F_{ij \gb}^{\ga}
\tfrac{(x-x_s)^i}{2 t_s} G dV}\\
 & \hsp- 2 \int_{\bRn}{ \left(  \dot{\gG}_{i \ga}^{\gd} \dot{\gG}_{j \gd}^{\gb}
F_{ij \gb}^{\ga} - \dot{\gG}_{i \gd}^{\gb} F_{ij \gb}^{\ga} \dot{\gG}_{j
\ga}^{\gd} \right) G dV}.
 \end{align*}
 Therefore in coordinate invariant form,
 \begin{align}\label{eq:sT1}
\int_{\bRn}{T_1 G_s dV}  & =  - 8 \int_{\bRn}{ \langle \ddot{\gG}_s, S_{x_s,t_s}
\rangle G_s dV} + 4 \int_{\bRn}{ \langle \dot{\gG}_s,  [\dot{\gG}_s,F_s]^{\#}
\rangle G_s dV}.
\end{align}
Next for the $T_2$ quantity we expand terms and then applying integration by
parts:
\begin{align}
\begin{split}\label{eq:sT2}
\int_{\bRn}{T_2 G_s dV} &= - 2\int_{\bRn}{D_i \dot{\gG}_{j \ga}^{\gb} D_i
\dot{\gG}_{j \gb}^{\ga} G dV} \\
&= - 2\int_{\bRn}{(\nab_i \dot{\gG}_{j \ga}^{\gb} - \nab_j \dot{\gG}_{i
\ga}^{\gb}) D_i \dot{\gG}_{j \gb}^{\ga} G dV} \\
&= - 2 \int_{\bRn}{\left( \nab_i \dot{\gG}_{j \ga}^{\gb} D_i \dot{\gG}_{j
\gb}^{\ga} \right) G dV} +  2 \int_{\bRn}{\left(  \nab_j \dot{\gG}_{i \ga}^{\gb}
D_i \dot{\gG}_{j \gb}^{\ga} \right) G dV} \\
&= - 2\int_{\bRn}{\left( \nab_i \dot{\gG}_{j \ga}^{\gb} D_i \dot{\gG}_{j
\gb}^{\ga} \right) G dV} + 2\int_{\bRn}{\left(  \nab_i \dot{\gG}_{j \ga}^{\gb}
D_j \dot{\gG}_{i \gb}^{\ga} \right) G dV} \\
&=2 \int_{\bRn}{ \dot{\gG}_{j \ga}^{\gb} \left(\nab_i  D_i \dot{\gG}_{j
\gb}^{\ga} - \left( \tfrac{x-x_s}{2 t_s} \right)^i D_i \dot{\gG}_{j \gb}^{\ga}
\right) G dV} \\
& \hsp -2 \int_{\bRn}{ \dot{\gG}_{j \ga}^{\gb} \left( \nab_i  D_j \dot{\gG}_{i
\gb}^{\ga} - \left( \tfrac{x-x_s}{2 t_s} \right)^i D_j \dot{\gG}_{i \gb}^{\ga}
\right) G dV}\\
&= 2 \int_{\bRn}{ \dot{\gG}_{j \ga}^{\gb} \left(\nab_i  \left( D_i \dot{\gG}_{j
\gb}^{\ga} - D_j \dot{\gG}_{i \gb}^{\ga} \right) + \left( \left( \tfrac{x-x_s}{2
t_s} \right)^i D_j \dot{\gG}_{i \gb}^{\ga} - \left( \tfrac{x-x_s}{2 t_s}
\right)^i D_i \dot{\gG}_{j \gb}^{\ga} \right) \right) G dV} \\
&=2\int_{\bRn}{ \dot{\gG}_{j \ga}^{\gb} \left(\nab_i  D_i \dot{\gG}_{j
\gb}^{\ga} - \left( \tfrac{x-x_s}{t_s} \right)^i (D_i \dot{\gG}_{j \gb}^{\ga}
)\right) G dV} \\
&= 4\int_{\bRn}{ \left\langle \dot{\gG}_s, D_s^* D_s \dot{\gG}_s + \left(
\tfrac{x-x_s}{2t_s} \right) \hook D_s \dot{\gG}_s \right\rangle G_s dV}.
\end{split}
\end{align}
Lastly we expand the $T_4$ quantity by expanding and integrating by parts:
\begin{align}
\begin{split}\label{eq:sT4}
\int_{\bRn}{T_4 G_s dV}
&= 4 \int_{\bRn}{ \left(\mathsf{g}_s \langle D_s \dot{\gG}_s, F_s \rangle
\right) G_s dV}\\
&= - 4 \int_{\bRn}{ \left( \mathsf{g} ( D_i \dot{\gG}_{j \ga}^{\gb}) F_{ij
\gb}^{\ga} \right) G dV}\\
&= -8 \int_{\bRn}{ \left( \mathsf{g} ( \nab_i \dot{\gG}_{j \ga}^{\gb} ) F_{ij
\gb}^{\ga} \right) G dV}\\
&= 8 \int_{\bRn}{  \left( (\del_i \mathsf{g})  \tfrac{(x-x_s)^i}{2 t_s} \right) 
 \dot{\gG}_{j \ga}^{\gb} F_{ij \gb}^{\ga} G dV} +8 \int_{\bRn}{  \mathsf{g}
\dot{\gG}_{j \ga}^{\gb}  (\nab_i F_{ij \gb}^{\ga}) G dV}\\
&= 8 \int_{\bRn}{\left\langle \dot{\gG}_s, \mathsf{g}_s \left( D_s^*F_s + \left(
\tfrac{x-x_s}{2 t_s} \hook F_s \right) - (\del \mathsf{g}_s) \hook F_s \right)
\right\rangle G_s dV}\\
&= 8  \int_{\bRn}{\left\langle \dot{\gG}_s, \mathsf{g}_s\left( S_{x_s,t_s}-
(\del \mathsf{g}_s) \hook F_s \right) \right\rangle G_s dV}.
\end{split}
\end{align}
Combining all quantities \eqref{eq:sT1}, \eqref{eq:sT2}, and \eqref{eq:sT4}
together into \eqref{eq:4part2ndvar} yields
\begin{align}
\begin{split}\label{eq:2ndvar}
\frac{d^2}{d s^2} \left[  \int_{\bRn} | F_s|^2 G_s dV \right]
 &= 8 \int_{\bRn}{ \langle \ddot{\gG}_s,S_{x_s,t_s} \rangle G_s dV} + 2
\int_{\bRn}{\left\langle  \dot{\gG}_s, \mathsf{g}_s S_{x_s,t_s} \right\rangle
G_s dV} + \int_{\bRn}{(\dot{\mathsf{g}}_s + \mathsf{g}_s^2) |F_s|^2 G_s dV}\\
& \hsp + 4\int_{\bRn}{ \left\langle \dot{\gG}_s, D_s^* D_s \dot{\gG}_s + \left(
\tfrac{x-x_s}{2t_s} \right) \hook D_s \dot{\gG}_s  + [\dot{\gG}_s,F_s]^{\#}
\right\rangle G_s dV}\\
& \hsp - 4 \int_{\bRn}{\left\langle \dot{\gG}_s , \left( \tfrac{\dot{t}_s}{t_s}
(x-x_s) + \tfrac{\dot{x}_s}{t_s} \right) \hook F_s \right\rangle G_s dV}.
\end{split}
\end{align}
Now we incorporate the temporal parameter and differentiate the quantity
\begin{align}
\begin{split}\label{eq:z2ndvart2Fexp}
\frac{d^2}{d s^2} \left[ \mathcal{F}_{x_s,t_s} \right] &= \frac{\del}{\del s}
\left[ 2 \dot{t}_s t_s \mathcal{F}_{x_s,t_s} + t_s^2 \frac{\del}{\del s} \left[
\mathcal{F}_{x_s,t_s} \right] \right]
= 2 (\ddot{t}_s t_s + \dot{t}_s^2 ) \mathcal{F}_{x_s,t_s} + 4 \dot{t}_s t_s
\frac{\del}{\del s}[\mathcal{F}_{x_s,t_s}]  + t_s^2 \frac{\del^2}{\del s^2}
\left[ \mathcal{F}_{x_s,t_s} \right].
\end{split}
\end{align} 
We insert the terms from \eqref{eq:2ndvar} and \eqref{eq:zYMfirstvar} into
\eqref{eq:z2ndvart2Fexp} and obtain
\begin{align*}
\begin{split}
\frac{d^2}{d s^2} \left[ \mathcal{F}_{x_s,t_s} \right]
 &= 2(\ddot{t}_s t_s + \dot{t}_s^2) \int_{\bRn}{|F_s|^2 G_s dV} + 4 \dot{t}_s
t_s \int_{\bRn}{| F_s |^2 \mathsf{g}_s G_s dV} + 8 \dot{t}_s t_s
\int_{\bRn}{\left\langle \dot{\gG}_s, S_{x_s,t_s}(\nab_s) \right\rangle G_s 
dV}\\
& \hsp  8 t_s^2 \int_{\bRn}{ \langle \ddot{\gG}_s, S_{x_s,t_s}(\nab_s) \rangle
G_s dV} + 2 t_s^2 \int_{\bRn}{\left\langle  \dot{\gG}_s, \mathsf{g}_s
S_{x_s,t_s}(\nab_s) \right\rangle G_s dV} + t_s^2
\int_{\bRn}{(\dot{\mathsf{g}}_s + \mathsf{g}_s^2) |F_s|^2 G_s dV}\\
& \hsp + 4 t_s^2 \int_{\bRn}{ \left\langle \dot{\gG}_s, D_s^* D_s \dot{\gG} +
\left( \frac{x-x_s}{2t_s} \right) \hook D_s \dot{\gG}_s  +
[\dot{\gG}_s,F_s]^{\#} \right\rangle G_s dV}\\
& \hsp -4 t_s^2 \int_{\bRn}{\left\langle \dot{\gG}_s, \left( \frac{\dot{t}_s}{
t_s} (x-x_s) + \frac{\dot{x}_s}{t_s} \right) \hook F_s \right\rangle G_s dV}\\
& = \int_{\bRn}{\left(2(\ddot{t}_s t_s + \dot{t}_s^2)  + t_s^2 (\mathsf{g}_s^2 +
\dot{\mathsf{g}}_s) + 4 \dot{t}_s t_s \mathsf{g}_s \right)|F_s|^2 G_s dV} \\
& \hsp + \int_{\bRn}{\left\langle \dot{\gG}_s, (8 \dot{t}_s t_s + 2t_s^2
\mathsf{g}_s) S_{x_s,t_s}(\nab_s) \right\rangle G_s  dV} + 8 t_s^2 \int_{\bRn}{
\langle \ddot{\gG}_s, S_{x_s,t_s}(\nab_s) \rangle G_s dV}\\
& \hsp + 4 t_s^2 \int_{\bRn}{ \left\langle \dot{\gG}_s, D_s^* D_s \dot{\gG}_s +
\left( \frac{x-x_s}{2t_s} \right) \hook D_s \dot{\gG}_s  + [\dot{\gG}_s,
F_s]^{\#} \right\rangle G_s dV}\\
& \hsp  -4 t_s \int_{\bRn}{\left\langle \dot{\gG}_s, \left( \dot{t}_s (x-x_s) +
\dot{x}_s \right) \hook F_s \right\rangle G_s dV}.
\end{split}
\end{align*}
The result follows.
\end{proof}
\end{prop}

We now specialize this result to the case of solitons. For notational clarity,
after evaluating at $s=0$, we excise the subscript except for those of the base
point $(x_0,t_0)$ and the heat kernel. 

\begin{cor}[Second Variation for Shrinkers]\label{cor:2ndvarshrink} Suppose that
$\nab$ is an $(x_0,t_0)$-soliton.  Then 
\begin{align*}
\begin{split}
\FF''_{x_0,t_0}(\dot{t},\dot{x},\dot{\N}) &= \left. \frac{d^2}{d s^2} \left[
\mathcal{F}_{x_s,t_s}(\nab_s) \right]
\right|_{s=0}\\
& = - 4  t_0 \dot{t}_0^2 \int_{\bRn}{|D^*F|^2 G_0 dV} - 2 t \int_{\bRn}{|F \hook
\dot{x}_0|^2 G_0 dV}\\
& \hsp + 4 t^2 \int_{\bRn}{ \left\langle \dot{\gG}, D^* D \dot{\gG} + \left(
\frac{x-x_0}{2t_0} \right) \hook D \dot{\gG}  + [\dot{\gG}, F]^{\#}
\right\rangle G_0 dV}\\
& \hsp -4 t_0 \int_{\bRn}{\left\langle \dot{\gG}, \left(\dot{t}_0 (x-x_0) +
\dot{x}_0 \right) \hook F \right\rangle G_0 dV}.
\end{split}
\end{align*}

\begin{proof} 
We first compute to identities concerning $\mathsf{g}_s$ to be used for the
following argument. First we have
\begin{align}
\begin{split}\label{eq:delgid}
\frac{\del \mathsf{g}}{\del s}  &= - \frac{n}{2} \left( \frac{\ddot{t}_s}{t_s} -
\frac{\dot{t}_s^2}{t_s^2} \right) + \frac{|x-x_s|^2}{4} \left(
\frac{\ddot{t}_s}{t^2_s} - \frac{2 \dot{t}_s^2}{t_s^3} \right) - \frac{\dot{t}_s
\langle \dot{x}_s, x-x_s \rangle}{t^2_s} + \frac{\langle \ddot{x}_s , x-x_s
\rangle}{2t_s} - \frac{|\dot{x}_s|^2}{2t_s}.
\end{split}
\end{align}
Then 
\begin{align}
\begin{split}\label{eq:g2id}
\mathsf{g}^2_s &= \left( - \frac{n \dot{t}_s}{2 t_s}+ \frac{\dot{t}_s |x -
x_s|^2}{4 t_s^2} + \frac{\langle \dot{x}_s, x - x_s \rangle}{2 t_s} \right)^2\\
&= \frac{n^2 \dot{t}_s^2}{4 t_s^2} - \frac{n \dot{t}_s^2 |x-x_s|^2}{4 t_s^3} -
\frac{n \dot{t}_s \langle \dot{x}_s , x-x_s \rangle}{2 t_s^2} + \frac{\dot{t}_s
|x-x_s|^2 \langle \dot{x}_s, x-x_s \rangle}{4 t_s^3} + \frac{\dot{t}_s^2
|x-x_s|^4}{16 t_s^4} + \frac{\langle \dot{x}_s, x - x_s \rangle^2}{4 t_s^2}.
\end{split}
\end{align}
Adding the two quantities \eqref{eq:delgid} and \eqref{eq:g2id} and labeling
with the corresponding items of Corollary \ref{5partlemma} we obtain
\begin{align*}
\dot{\mathsf{g}} + \mathsf{g}^2&= 
\left(  - \frac{n \ddot{t}_0}{2 t_0} + \frac{n \dot{t}_0^2}{2 t_0^2} + \frac{n^2
\dot{t}_0^2}{4 t_0^2} \right) + \left( \frac{|x-x_0|^2}{4 t_0} \left(
\frac{\ddot{t}_0}{t_0} - \frac{n \dot{t}_0^2}{t_0^2} - \frac{2
\dot{t}_0^2}{t_0^2} \right) \right)_a + \left( \langle \dot{x}_0, x-x_0 \rangle
\left( - \frac{\dot{t}_0}{t^2_0} - \frac{n \dot{t}_0}{2 t_0^2} \right) \right)_b
\\
& \hsp  + \left( |x-x_0|^2 \langle \dot{x}_0, x-x_0 \rangle \left(
\frac{\dot{t}_0}{4 t_0^3} \right) \right)_d + \left( \langle \ddot{x}_0, x-x_0
\rangle \frac{1}{2t_0} \right)_b + \left( |\dot{x}_0|^2 \left( \frac{-1}{2t_0}
\right) \right)\\
& \hsp + \left( |x-x_0|^4 \left( \frac{\dot{t}_0^2 }{16 t_0^4} \right) \right)_c
+ \left( \langle \dot{x}_0, x - x_0 \rangle^2 \left(  \frac{1}{4 t_0^2} \right)
\right)_e.
\end{align*}
Applying the said identities of Corollary \ref{5partlemma} we obtain
\begin{align*}
\int_{\bRn}{\left( \dot{\mathsf{g}} +\mathsf{g}^2 \right) |F|^2 G_s dV} 
&= \int_{\bRn}{\left( \left(  - \frac{n \ddot{t}_0}{2 t_0} + \frac{n
\dot{t}_0^2}{2 t_0^2} + \frac{n^2 \dot{t}_0^2}{4 t_0^2} \right) +  \left(
\frac{n-4}{2} \left( \frac{\ddot{t}_0}{t_0} - \frac{n \dot{t}_0^2}{t_0^2} -
\frac{2 \dot{t}_0^2}{t_0^2} \right) \right) \right) |F|^2 G_0 dV}\\
&\hsp + \int_{\bRn}{\left( \left( |\dot{x}_0|^2 \left( \frac{-1}{2t_0} \right)
\right) + 4(n-2)(n-4) t_0^2  \left( \frac{\dot{t}_0^2}{16 t_0^4} \right) +
\left( \frac{1}{2 t_0} \right)  | \dot{x}_0 |^2 \right) |F|^2 G_0 dV}\\
& \hsp - 4 \left( \frac{\dot{t}_0^2}{t_0} \right) \int_{\bRn}{|D^*F|^2 G_0 dV} -
8 t_0 \left( \frac{1}{4 t_0^2} \right) \int_{\bRn}{|F \hook \dot{x}_0|^2 G_0
dV}\\
&= \int_{\bRn}{\left( \left(  - \frac{n \ddot{t}_0}{2 t_0} + \frac{n
\dot{t}_0^2}{2 t_0^2} + \frac{n^2 \dot{t}_0^2}{4 t_0^2} \right) +  \left(
\frac{n-4}{2} \left( \frac{\ddot{t}_0}{t_0} - \frac{n \dot{t}_0^2}{t_0^2} -
\frac{2 \dot{t}_0^2}{t_0^2} \right) \right) \right) |F_s|^2 G_0 dV}\\
&\hsp + \int_{\bRn}{\left( (n-2)(n-4)  \left( \frac{\dot{t}_0^2}{4 t_0^2}
\right) \right) |F|^2 G_0 dV}\\
& \hsp - 4 \left( \frac{\dot{t}_0^2}{t_0} \right) \int_{\bRn}{|D^*F|^2 G_0 dV} -
8 t_0 \left( \frac{1}{4 t_0^2} \right) \int_{\bRn}{|F \hook \dot{x}_0|^2 G_0
dV}.
\end{align*}
We collect and simplify the coefficients of the integrands multiplied against
$|F|^2 G$ of the first and second line to obtain
\begin{align*}
\left( \frac{4 \dot{t}_0^2}{ t_0^2} - \frac{2\ddot{t}_0}{t_0} + \frac{2
\dot{t}_0^2}{t_0^2} \right) &+ n \left( - \frac{\ddot{t}_0}{2 t_0} +
\frac{\dot{t}_0^2}{2 t_0^2} +  \frac{2 \dot{t}_0^2}{t_0^2} + \frac{\ddot{t}_0}{2
t_0} - \frac{\dot{t}_0^2}{t_0^2} - \frac{3 \dot{t}_0^2}{2 t_0^2} \right) + n^2
\left( \frac{\dot{t}_0^2 }{4 t^2_0} - \frac{\dot{t}_0^2}{2 t_0^2} +
\frac{\dot{t}_0^2}{4 t_0^2} \right) \\
&  =\frac{6 \dot{t}_0^2}{ t_0^2} - \frac{2\ddot{t}_0}{t_0}.
\end{align*}
Therefore we conclude that
\begin{align*}
t_0^2 \int_{\bRn}{\left( \dot{\mathsf{g}} +\mathsf{g}^2 \right) |F|^2 G dV} &=
\int_{\bRn}{\left( 6\dot{t}^2_0 - 2 \ddot{t}_0 t_0 \right) |F|^2 G dV} - 4  t_0
\dot{t}_0^2 \int_{\bRn}{|D^*F|^2 G dV} - 2 t_0\int_{\bRn}{|F \hook \dot{x}_0|^2
G dV}.
\end{align*}
We combine terms, apply Corollary \eqref{5partlemma} (b) and then (a),
\begin{align*}
\int_{\bRn}&{\left(2(\ddot{t}_0 t_0 + \dot{t}_0^2)  + t_0^2 (\mathsf{g}^2 +
\dot{\mathsf{g}}) + 4 \dot{t}_0 t_0 \mathsf{g} \right)|F|^2 G_0 dV} \\
&= \int_{\bRn}{\left(2 \dot{t}_0^2  -2 n \dot{t}_0^2 + \frac{\dot{t}_0^2 |x -
x_0|^2}{ t_0}+ \dot{t}_0 \langle \dot{x}_0, x - x_0 \rangle \right)  |F|^2 G_0
dV}  + \int_{\bRn}{ 6 \dot{t}_0^2 |F|^2 G_0 dV}\\
& \hsp - 4  t_0 \dot{t}_0^2 \int_{\bRn}{|D^*F|^2 G_0 dV} - 2 t_0\int_{\bRn}{|F
\hook \dot{x}_0|^2 G_0 dV}\\
&=- 4  t_0 \dot{t}_0^2 \int_{\bRn}{|D^*F|^2 G_0 dV} - 2 t_0\int_{\bRn}{|F \hook
\dot{x}_0|^2 G_0 dV}.
\end{align*}
Therefore we conclude that
\begin{align*}
\begin{split}
\left. \frac{d^2}{d s^2} \left[ \mathcal{F}_{x_s,t_s}(\nab_s) \right]
\right|_{s=0} & = - 4  t_0 \dot{t}_0^2 \int_{\bRn}{|D^*F|^2 G_0 dV} - 2 t_0
\int_{\bRn}{|F \hook \dot{x}_0|^2 G_0 dV}\\
& \hsp + 4 t_0^2 \int_{\bRn}{ \left\langle \dot{\gG}, D^* D \dot{\gG} + \left(
\frac{x-x_0}{2t_0} \right) \hook D \dot{\gG}  + [\dot{\gG}, F]^{\#}
\right\rangle G_0 dV}\\
& \hsp -4 t_0^2 \int_{\bRn}{\left\langle \dot{\gG}, \left( \dot{t}_0 (x-x_0) +
\dot{x}_0 \right) \hook F \right\rangle G_0 dV}.
\end{split}
\end{align*}
The result follows.
\end{proof}
\end{cor}

\section{\texorpdfstring{$\mathcal F$}{F}-stability and gap theorem}
\label{stability}

In this section we establish a criterion for checking entropy stability (Theorem
\ref{stabthm}) and also prove a gap theorem for shrinkers (Theorem
\ref{gapthm}).  The basic observation behind the stability condition is that for
a shrinker, the second variation operator $L$ \emph{always} has negative
eigenvalues, one corresponding to the Yang-Mills flow direction itself, which is
the same as moving in time while scaling, the other corresponding to translation
in space.  This in some sense is the whole reason for explicitly including the
space and time parameters in the definition of entropy, as we then show in
Theorem \ref{stabthm} that these directions can be accounted for by an
appropriate choice of variation in the basepoint.

Recalling Corollary \ref{cor:2ndvarshrink} we define the operator $L_{x_0,t_0}$
by
\begin{align*}
\begin{split}
L_{x_0,t_0}: \Lambda^1(\End E) \to \Lambda^1(\End E): B \mapsto D^* D B + \left(
\frac{x-x_0}{2t_0} \right) \hook D B  + [B, F]^{\#}.\\
\end{split}
\end{align*}
In particular, we set $L := L_{0,1}$. By the Bochner formula this is equal to
\begin{align*}
\begin{split}
L_{x_0,t_0}(B) = - \lap B - \N D^* B - 2 \left[B ,F \right]^{\#}.
\end{split}
\end{align*}
We ultimately only want to apply $L$ to elements of an appropriate weighted
Sobolev space.  For a given $\nab \in \mathfrak{S}$ set
\begin{equation} \label{Wdef}
W_{\N}^{2,2} := \left\{ B \in \Lambda^{1}(\End E) : \int_{\bRn}{\left(|B|^2 +
|\nab B |^2 + |LB|^2 \right) e^{\frac{-|x|^2}{4}}dV} < \infty\right\} . 
\end{equation}

\subsection{Second variation operator}

\begin{defn} \label{stabdefn} A soliton $\N$ is called $\FF$-stable if for any
$B \in W_{\N}^{2,2}$ there exist a real number $\gs$ and a constant vector field
$V$ such that $\mathcal F''_{0,1}(q,V,B) \geq 0$.
\end{defn}

\begin{lemma}\label{lem:VhashFgen}
Let $V = V^i \del_i$ be some vector field. Then
\begin{align}
\begin{split}\label{eq:VhashFgen}
L(V\hook F)_{k \ga}^{\gb}& = (\N_p \N_k V^i) F_{ip \ga}^{\gb} -(\N_p \N_p V^i)
F_{ik\ga}^{\gb}- \tfrac{V^i }{2} F_{i k \ga}^{\gb} \\
& \hsp - (\N_p V^i )\left( (\N_p F_{ik\ga}^{\gb})  + (\N_i F_{pk \ga}^{\gb} )  -
\tfrac{x_p}{2} F_{ik\ga}^{\gb}  \right).
\end{split}
\end{align}

\begin{proof}
We compute the first term of the $L$ operator and obtain
\begin{align}
\begin{split}\label{eq:DD*VhashF}
D^*D \lbr V \hook F \rbr_{k \ga}^{\gb}
&= - \N_p D_p \lbr V \hook F \rbr\\
& = - \N_p \N_p \lbr V^i F_{ik\ga}^{\gb} \rbr + \N_p \N_k \lbr V^i F_{ip
\ga}^{\gb} \rbr \\
&= - \N_p \lbr (\N_p V^i)F_{ik\ga}^{\gb} + V^i (\N_p F_{ik\ga}^{\gb})  - (\N_k
V^i ) F_{ip \ga}^{\gb} - V^i (\N_k F_{ip \ga})\rbr \\
& = (\N_p \N_k V^i) F_{ip \ga}^{\gb} -(\N_p \N_p V^i) F_{ik\ga}^{\gb} +
\left[V^i (\N_p \N_k F_{ip\ga}^{\gb}) - V^i ( \N_p \N_p F_{ik\ga}^{\gb})
\right]_T\\
& \hsp - 2(\N_p V^i )(\N_p F_{ik\ga}^{\gb})+ (\N_p V^i)(\N_k F_{ip \ga}^{\gb}) +
(\N_k V^i)( \N_p F_{ip\ga}^{\gb}).
\end{split}
\end{align}
Now we manipulate $T$. Observe that
\begin{align*}
T &= - V^i \left( \N_p \N_p   F_{i k \ga}^{\gb}) - \N_p \N_k F_{i p \ga}^{\gb}
\right)\\
&= - V^i \left( - \N_p  \lbr \N_k  F_{p i \ga}^{\gb} + \N_i   F_{k p \ga}^{\gb}
\rbr - \N_p \N_k F_{i p \ga}^{\gb} \right)\\
&=  V^i \left(  \N_p \N_i   F_{k p \ga}^{\gb}\right)\\
&=- V^i  (\N_i \N_p  F_{pk \ga}^{\gb}) + V_i [\N_p, \N_i ]F_{kp \ga}^{\gb}\\
&= - V^i  (\N_i \N_p   F_{pk \ga}^{\gb}) + V_i (F_{pi \gd}^{\gb}F_{kp \ga}^{\gd}
+ F_{pi \ga}^{\gd}F_{kp \gd}^{\gb})\\
&= V^i  (\N_i (D^*F)_{k \ga}^{\gb}) -  V_i (F_{ip \gd}^{\gb}F_{kp \ga}^{\gd} -
F_{ip \ga}^{\gd}F_{kp \gd}^{\gb})\\
&= - V^i  \N_i \lbr \tfrac{x^p}{2} F_{p k \ga}^{\gb} \rbr + ( [F, (V \hook
F)]^{\#})_{k \ga}^{\gb}\\
&=-  \tfrac{ V^i }{2} \left( F_{i k \ga}^{\gb} - x_p \N_i F_{p k \ga}^{\gb}
\right) + ( [F, (V \hook F)]^{\#})_{k \ga}^{\gb}\\
&=- \tfrac{V^i }{2} F_{i k \ga}^{\gb} + \tfrac{V_i }{2} \left( x_p (\N_k F_{ ip
\ga}^{\gb} + \N_p F_{k i \ga}^{\gb} ) \right) + ( [F, (V \hook F)]^{\#})_{k
\ga}^{\gb} \\
& = - \tfrac{V^i }{2} F_{i k \ga}^{\gb} + \tfrac{V_i }{2} x_p \N_k F_{ ip
\ga}^{\gb} - \tfrac{V_i }{2}  x_p\N_p F_{ ik \ga}^{\gb} + ( [F, (V \hook
F)]^{\#})_{k \ga}^{\gb}.
\end{align*}
Therefore we conclude that, applying the identity of $T$ to
\eqref{eq:DD*VhashF},
\begin{align}
\begin{split}\label{eq:DD*VhashF1}
D^*D \left[ V \hook F \right]_{k \ga}^{\gb} &=  (\N_p \N_k V^i) F_{ip \ga}^{\gb}
-(\N_p \N_p V^i) F_{ik\ga}^{\gb} \\
& \hsp - 2(\N_p V^i )(\N_p F_{ik\ga}^{\gb})+ (\N_p V^i)(\N_k F_{ip \ga}^{\gb}) +
(\N_k V^i)( \N_p F_{ip\ga}^{\gb})\\
&\hsp - \tfrac{V^i }{2} F_{i k \ga}^{\gb} + \tfrac{V_i }{2} x_p \N_k F_{ ip
\ga}^{\gb} - \tfrac{V_i }{2}  x_p\N_p F_{ ik \ga}^{\gb} + ( [F, (V \hook
F)]^{\#})_{k \ga}^{\gb}.
\end{split}
\end{align}
Next we compute
\begin{align}
\begin{split}\label{eq:halfxhashDVhashF}
\tfrac{x}{2} \hook D \lbr V \hook F \rbr
& = \frac{x_p}{2} D_p \lbr V^i F_{ik\ga}^{\gb}\rbr \\
& = \frac{x_p}{2} \left( \N_p \lbr V^i F_{ik\ga}^{\gb} \rbr - \N_k \lbr V^i
F_{ip\ga}^{\gb} \rbr \right) \\
& = \frac{x_p}{2} \left( (\N_p V^i ) F_{ik\ga}^{\gb} + V^i  (\N_p
F_{ik\ga}^{\gb} ) - (\N_k V^i) F_{ip\ga}^{\gb} - V^i (\N_k F_{ip\ga}^{\gb})
\right).
\end{split}
\end{align}
Therefore we have that, applying \eqref{eq:DD*VhashF1} and
\eqref{eq:halfxhashDVhashF} to the formula for $L(V \hook F)$, we have
\begin{align*}
L(V \hook F)_{k \ga}^{\gb} &= (D^*D (V \hook F))_{k \ga}^{\gb} + \tfrac{x}{2}
\hook \left( D \lbr V \hook F \rbr \right)_{k \ga}^{\gb} + \left(\left[ V \hook
F, F \right]^{\#} \right)_{k\ga}^{\gb}\\
& = (\N_p \N_k V^i) F_{ip \ga}^{\gb} -(\N_p \N_p V^i) F_{ik\ga}^{\gb} \\
& \hsp - 2(\N_p V^i )(\N_p F_{ik\ga}^{\gb})+ (\N_p V^i)(\N_k F_{ip \ga}^{\gb}) +
(\N_k V^i)( \N_p F_{ip\ga}^{\gb})\\
&\hsp - \tfrac{V^i }{2} F_{i k \ga}^{\gb} + \tfrac{x_p}{2} (\N_p V^i )
F_{ik\ga}^{\gb} -\tfrac{x_p}{2}  (\N_k V^i) F_{ip\ga}^{\gb} \\
& = (\N_p \N_k V^i) F_{ip \ga}^{\gb} -(\N_p \N_p V^i) F_{ik\ga}^{\gb} \\
& \hsp - 2(\N_p V^i )(\N_p F_{ik\ga}^{\gb})+ (\N_p V^i)(\N_k F_{ip \ga}^{\gb}) +
(\N_k V^i)( DF)_{i\ga}^{\gb} \\
&\hsp - \tfrac{V^i }{2} F_{i k \ga}^{\gb} + \tfrac{x_p}{2} (\N_p V^i )
F_{ik\ga}^{\gb} + \tfrac{x_p}{2}  (\N_k V^i) F_{pi\ga}^{\gb} \\
& = (\N_p \N_k V^i) F_{ip \ga}^{\gb} -(\N_p \N_p V^i) F_{ik\ga}^{\gb} \\
& \hsp - (\N_p V^i )(\N_p F_{ik\ga}^{\gb}) + (\N_p V^i )\left((\N_p F_{ki
\ga}^{\gb}) +(\N_k F_{ip \ga}^{\gb}) \right)\\
&\hsp - \tfrac{V^i }{2} F_{i k \ga}^{\gb} + \tfrac{x_p}{2} (\N_p V^i )
F_{ik\ga}^{\gb} \\
& = (\N_p \N_k V^i) F_{ip \ga}^{\gb} -(\N_p \N_p V^i) F_{ik\ga}^{\gb}-
\tfrac{V^i }{2} F_{i k \ga}^{\gb} \\
& \hsp - (\N_p V^i )\left( (\N_p F_{ik\ga}^{\gb})  + (\N_i F_{pk \ga}^{\gb} )  -
\tfrac{x_p}{2} F_{ik\ga}^{\gb}  \right).
\end{align*}
The result follows.
\end{proof}
\end{lemma}

\begin{lemma}[Eigenforms of $L$]\label{leigenfunc}
For $\nab \in \mathfrak{S}$, and a constant vector field $V$,
\begin{equation}\label{eq:leigenfunc2}
L\left( V \hook F \right) = - \tfrac{1}{2}\left( V \hook  F \right).
\end{equation}
Furthermore
\begin{equation}\label{eq:leigenfunc1} L(D^*F) = - D^*F, \end{equation}

\begin{proof} The identity \eqref{eq:leigenfunc2} is an immediate corollary of
Lemma \ref{lem:VhashFgen}, by simply evaluating \eqref{eq:VhashFgen} on a
constant vector field $V= V^i \del_i$. For the identity \eqref{eq:leigenfunc1},
we compute the following, applying the first Bianchi identity to obtain
\begin{align*}
\left( D^*D(D^* F) \right)_{r \ga}^{\gb} &= - \left( D^*D \left( \tfrac{x}{2}
\hook\ F \right) \right)_{r \ga}^{\gb}\\
&= \nab_i \left( D_i \left( \tfrac{x_p}{2} F_{pr \ga}^{\gb} \right) \right)\\
&= \tfrac{1}{2} \nab_i \left( \nab_i \left( x_p F_{pr \ga}^{\gb} \right) -
\nab_r \left( x_p F_{pi \ga}^{\gb} \right) \right)\\
&= \tfrac{1}{2} \nab_i \left( F_{ir \ga}^{\gb}  + x_p \nab_i F_{pr \ga}^{\gb} -
F_{ri \ga}^{\gb} - x_p \nab_r F_{pi \ga}^{\gb} \right)\\ 
&= \tfrac{1}{2} \nab_i \left( F_{ir \ga}^{\gb} - F_{ri \ga}^{\gb} \right) +
\tfrac{1}{2} \nab_i \left( x_p \nab_i F_{pr \ga}^{\gb} - x_p \nab_r F_{pi
\ga}^{\gb}\right)\\
&=  - (D^*F)_{r \ga}^{\gb}  + \tfrac{1}{2} \nab_i \left( x_p \nab_i F_{pr
\ga}^{\gb} + x_p (\nab_i F_{rp \ga}^{\gb} + \nab_{p}F_{ir \ga}^{\gb}) \right)\\
&=  - (D^*F)_{r \ga}^{\gb}  + \tfrac{1}{2} \nab_i \left( x_i \nab_{p}F_{ir
\ga}^{\gb} \right)\\
& = - (D^*F)_{r \ga}^{\gb} + \tfrac{1}{2}\left( \delta_{ip} \nab_p F_{ir
\ga}^{\gb}+ x_p \nab_i \nab_p F_{ir \ga}^{\gb} \right)\\
&= - (D^*F)_{r \ga}^{\gb} + \tfrac{1}{2}\left( \nab_i F_{ir \ga}^{\gb}+
\tfrac{x_p}{2} \nab_i \nab_p F_{ir \ga}^{\gb} \right)\\
&= -(D^*F)_{r \ga}^{\gb} - \tfrac{1}{2} D^*F_{r \ga}^{\gb} + \tfrac{x_p}{2}
\nab_p \nab_i F_{i r \ga}^{\gb} + \tfrac{x_p}{2} [\nab_i, \nab_p] F_{i r
\ga}^{\gb}\\
&= - \tfrac{3}{2} (D^*F)_{r \ga}^{\gb}+ \tfrac{x_p}{2} D_p \nab_i F_{ir
\ga}^{\gb} + \tfrac{x_p}{2} \nab_r \nab_i F_{ip\ga}^{\gb} + \tfrac{ x_p}{2}
\left( F^{\gb}_{ip \gd} F_{ir \ga}^{\gd} - F^{\gd}_{ip \ga} F_{ir \gd}^{\gb}
\right)\\
&= - \tfrac{3}{2} (D^*F)_{r \ga}^{\gb}  - \left(  \tfrac{x}{2} \hook DD^*F
\right)_{r \ga}^{\gb}  + \tfrac{x_p}{2} \nab_r \nab_i F_{ip\ga}^{\gb} + \left(
[F, D^*F]^{\#} \right)_{r \ga}^{\gb}.
\end{align*}
We simplify the third term, nothing vanishing due to the product of skew and
symmetric matrices
\begin{align*}
\tfrac{x_p}{2} \nab_r \nab_i F_{ip\ga}^{\gb} &= - x_p \nab_r \left(
\tfrac{x_s}{2} F_{s p \ga}^{\gb} \right)\\
&= \left( - \tfrac{x_p }{2} \delta_{rs} F_{s p \ga}^{\gb} - x_s x_p \nab_r F_{sp
\ga}^{\gb} \right)\\
&= - x_p F_{r p \ga}^{\gb} \\
&= \tfrac{1}{2} (D^*F)_{r \ga}^{\gb}.
\end{align*}
Applying this to the above computation we conclude that
\begin{equation*}
(D^* D (D^* F))_{r \ga}^{\gb} = - (D^*F)_{r \ga}^{\gb}  - \left(  \tfrac{x}{2}
\hook DD^*F \right)_{r \ga}^{\gb}  + \left( [F, D^*F]^{\#} \right)_{r
\ga}^{\gb}.
\end{equation*}
Then rearranging the equality we have \eqref{eq:leigenfunc1}, as desired. The
results follow.
\end{proof}
\end{lemma}

\begin{defn} Given $\N \in \mathfrak{S}$ and $\gl \in \mathbb R$, let
\begin{equation*}
\chi_{\la} := \left\{ B \in \Lambda^1(M) : L B = \la B \right\}.
\end{equation*}
\end{defn}

\begin{thm} \label{stabthm} Let $\nab \in \mathfrak{S}$ have polynomial energy
growth.  Then $\nab$ is $\mathcal{F}$-stable if and only if one has the
conditions:

\begin{enumerate}
\item $\chi_{-1} = \{ \rho D^*F  : \rho \in \mathbb{R} \}$,
\item $\chi_{-1/2} = \{ V \hook F : V \in \mathbb{R}^n \}$,
\item $\chi_{\la} = \{ 0 \}$ for any $\la < 0$ and $\la \notin \{ -1,
\frac{-1}{2} \}$.
\end{enumerate}

\begin{proof} Fix $B \in \Lambda^1(E)$ and decompose it as
\begin{equation}
B :=  \varsigma D^*F + (\varrho \hook F) + \varpi,
\end{equation}
where $\varsigma \in \mathbb{R}$, $\varrho \in \mathbb{R}^n$, and $\varpi \in
\Lambda^1(E)$, and, for all $V \in \mathbb{R}^n$,
\begin{equation}
\int_{\mathbb{R}^n}{ \langle \varpi, D^*F \rangle G dV} =
\int_{\mathbb{R}^n}{\langle \varpi, (V \hook F) \rangle G dV} = 0.
\end{equation}
Using Corollary \ref{cor:2ndvarshrink} and identities from Corollary
\ref{5partlemma} we have
\begin{align}
\begin{split}
\left. \frac{d^2}{d s^2} \left[ \mathcal{F}_{0,1}(\nab_s) \right] \right|_{s=0}
& = - 4  \dot{t}^2 \int_{\bRn}{|D^*F|^2 G dV} - 2 \int_{\bRn}{|F \hook
\dot{x}|^2 G dV} + 4 \int_{\bRn}{ \left\langle B, L(B)\right\rangle G dV}\\
& \hsp - 4 \int_{\bRn}{\left\langle B, - 2 \dot{t} D^*F +  \dot{x} \hook F 
\right\rangle G dV}\\
& = - 4  \dot{t}^2 \int_{\bRn}{|D^*F|^2 G dV} - 2 \int_{\bRn}{|F \hook
\dot{x}|^2 G dV} \\
& \hsp + 4 \int_{\bRn}{ \left\langle \varsigma D^*F + (\varrho \hook F) +
\varpi, -\varsigma D^*F - \tfrac{1}{2}(\varrho \hook F) + L(\varpi)
\right\rangle G dV}\\
& \hsp - 4 \int_{\bRn}{\left\langle \varsigma D^*F + (\varrho \hook F) + \varpi,
- 2 \dot{t} D^*F +  \dot{x} \hook F  \right\rangle G dV} \\
& = - 4  \dot{t}^2 \int_{\bRn}{|D^*F|^2 G dV} - 2 \int_{\bRn}{|F \hook
\dot{x}|^2 G dV} \\
& \hsp - 4 \varsigma^2 \int_{\bRn}{ |D^*F|^2 G dV} - 2 \int_{\bRn}{|\varrho
\hook F|^2 G dV} + 4 \int_{\bRn}{\langle \varpi, L(\varpi)\rangle dV}\\
& \hsp + 8 \varsigma \dot{t} \int_{\bRn}{|D^*F|^2 G dV} - 4 \int_{\bRn}{\langle
\varrho \hook F , \dot{x} \hook F \rangle G dV}\\
& = - 4  (\dot{t} - \varsigma)^2  \int_{\bRn}{|D^*F|^2 G dV} - 2 \int_{\bRn}{|F
\hook (\varrho + \dot{x})|^2 G dV} + 4 \int_{\bRn}{\langle \varpi,
L(\varpi)\rangle dV}.\end{split}
\end{align}
Choosing $\varsigma = \dot{t}$ and $\varrho = - \dot{x}$, we have that 
\begin{equation*}
\left. \frac{d^2}{d s^2} \left[ \mathcal{F}_{0,1}(\nab_s) \right] \right|_{s=0}
= 4 \int_{\bRn}{\langle \varpi, L(\varpi)\rangle dV}.
\end{equation*}
Both directions of the theorem follow from this calculation.
\end{proof}
\end{thm}

\subsection{Gap theorem}

In this subsection we establish Theorem \ref{gapthm}.  To begin we prove a lemma
showing that self-shrinking Yang-Mills connections are flat.

\begin{lemma}\label{selfsimYM} Suppose $\N$ is a soliton with bounded curvature
satisfying $D_{\N}^* F_{\N} = 0$.  Then $F_{\N} = 0$.
\end{lemma}

\begin{proof} 
Suppose to the contrary that $\N$ is a nonflat soliton and Yang-Mills
connection. As a soliton, by Proposition \ref{prop:sssiffgtfm}, there exists a
gauge such that for all $\la \in \mathbb{R}$, the connection coefficient matrix
satisfies $\gG(x,t) = \la \gG(\la x,\la^2 t)$.
Thus the curvature scales as
\begin{equation*}
F_{\N}(x,t) = \la^2 F_{\N}(\la x,\la^2 t),
\end{equation*}
with the given connection $\N$ as the time $-1$ slice.  Using this we note that
because $\N$ is nontrivial there exists some $y \in \bRn$ at which the following
limit holds:
\begin{equation}\label{eq:limofF}
\lim_{t \to 0} |F_{\N}(y \sqrt{-t},t)| = \lim_{t \to 0} \tfrac{1}{t} \left|
F_{\N}\left( y,-1\right) \right|= \infty.
\end{equation}
In particular we have $\sup_{x \in \mathbb{R}^n, {t \in [-1,0)}} | F_{\N}(x,t)|
= \infty$. Simultaneously, since $D^*_{\N} F_{\N} = 0$ and solutions to the
Yang-Mills flow on $\mathbb R^n$ with bounded curvature are unique, we obtain
that $\dt \gG = 0$ for all $(x,t) \in \mathbb R^n \times [-1,0)$.  Therefore we
have that for all $x \in \bRn$, then $|F_{\N}(x,t) | = |F_{\N}(x,-1)|$. This
implies $\sup_{x \in \mathbb{R}^n,{t \in [-1,0)}} | F_{\N}(x,t) | = \sup_{x \in
\mathbb{R}^n} | F_{\N}(x,-1) |  < \infty$, a contradiction. The result follows.
\end{proof}

\begin{proof} [Proof of Theorem \ref{gapthm}] Since $|F|$ is bounded, it follows
from local smoothing estimates for Yang-Mills flow (\cite{Weinkove} Theorem 2.2)
that $| D^* F|$ and $| \nab D^*F|$ are also uniformly bounded.  Integration by
parts yields that
\begin{align*} - \int_{\bRn}{ |\nab D^*F|^2 G dV} 
&= \int_{\bRn}{ (\nab_i (D^* F)_{ \ell \gd}^{\gb})( \nab_i (D^* F)_{\ell
\gb}^{\gd}) G dV} \\
&= -\int_{\bRn}{ (D^* F )_{\ell \gd}^{\gb} (\nab_i \nab_i (D^* F)_{\ell
\gb}^{\gd} G dV} - \int_{\bRn}{ (D^* F)_{\ell \gd}^{\gb}) (\nab_i (D^* F)_{\ell
\gb}^{\gd}) \nab_i G dV} \\
&= -\int_{\bRn}{(D^* F)_{\ell \gd}^{\gb} ( \lap (D^* F)_{\ell \gb}^{\gd}) G dV}
+ \int_{\bRn}{ (D^* F)_{\ell \gd}^{\gb} \nab_i (D^* F)_{\ell \gb}^{\gd} 
\frac{x_i}{2} G dV}.
\end{align*}
Thus from the Bochner formula we have
\begin{align*} \int_{\bRn}{ |\nab D^*F|^2 G dV}  &= \int_{\bRn}{\left\langle
D^*F ,  -\lap D^* F + \tfrac{x}{2} \hook \nab D^*F \right\rangle G dV}\\
&= \int_{\bRn}{\left\langle D^*F ,   \lap_DD^*F + [F,D^*F]^{\#} + \tfrac{x}{2}
\hook \nab D^*F \right\rangle G dV}.
\end{align*}
Note by Lemma \ref{D*D*2form} that $D^*D^*F = 0$.  Applying the definition of
$L$ to the above expression we obtain that, since $D^*F$ is an eigenfunction of
$L$ by Lemma \ref{leigenfunc},
\begin{align*} \int_{\bRn}{ |\nab D^*F|^2 G dV}  &= \int_{\bRn}{\left\langle
D^*F ,  L(D^*F) + 2[F,D^*F]^{\#} - \left( \tfrac{x}{2} \hook D D^* F \right) +
\tfrac{x}{2} \hook \nab D^*F \right\rangle G dV}\\
&= \int_{\bRn}{\left\langle D^*F ,  - D^*F + 2[F,D^*F]^{\#} -\left( \tfrac{x}{2}
\hook D D^* F \right) + \tfrac{x}{2} \hook \nab D^*F \right\rangle G dV}.
\end{align*}

In particular, we focus on simplifying the expression $\left(x \hook \nab D^*F-
x \hook D D^*F \right)$. This can be simplified by introducing a divergence term
and applying the first Bianchi identity, and Lemma \ref{D*D*2form} once more,
\begin{align*} \left(x \hook \nab D^*F- x \hook D D^*F \right)_{m \ga}^{\gb} &=
x_i \nab_i (D^*F)_{m \ga}^{\gb} - x_i D_i (D^*F)_{m \ga}^{\gb} \\
&=  x_i \nab_m (D^*F)_{i \ga}^{\gb} \\
&=  -x_i \nab_m \nab^k F_{ki \ga}^{\gb} \\
&= -  \nab_m \left( x_i \nab^k F_{ki \ga}^{\gb} \right) +  \nab^k F_{km
\ga}^{\gb} \\
&= - \nab_m \left[ \N^k (x_i F_{k i \ga}^{\gb}) - F_{kk \ga}^{\gb} \right] -
(D^*F)_{m \ga}^{\gb}\\
&=  - \nab_m \nab^k (D^*F)_{k \ga}^{\gb} - (D^*F)_{m \ga}^{\gb}\\
&=- (D^*F)_{m \ga}^{\gb}.
 \end{align*}
Inserting this above and applying the Cauchy Schwartz inequality yields
\begin{align*}
\int_{\bRn}{ |\nab D^*F|^2 G dV} &=-\tfrac{3}{2}\int_{\bRn}{ |D^*F|^2 G dV} +
2\int_{\bRn}{\langle D^*F, [D^*F, F]^{\#} \rangle G dV} \\
& \leq \left( 4 |F| -\tfrac{3}{2} \right) \int_{\bRn}{ |D^*F|^2 G dV}.
\end{align*}
Therefore for $\nab$ with $|F_{\nab}| \leq  \tfrac{3}{8}$, we have that $\left\|
(\nab D^*F) \sqrt{G} \right\|_{L^2} = 0$, which implies that 
$D^*F$ is parallel. Since $\nab$ is a soliton we have $D^*F = \tfrac{x}{2} \hook
F$, then at $x=0$, $D^*F$ vanishes and thus since $D^*F$ is parallel, $D^*F = 0$
for all $x$. By Lemma \ref{selfsimYM} it follows that $\nab$ is flat, as
desired.
\end{proof}

\section{Entropy Stability}\label{ss:Ffunctnentropy}

In this section we combine the results of the previous sections to establish
that for a soliton with polynomial energy growth the entropy is achieved at
$(0,1)$, and moreover it is uniquely achieved at this point unless the
connection has flat directions.  The strategy is very similar to (\cite{CM}
Lemma 7.10).  This culminates in the proof of Theorem \ref{entequiv}.

\begin{defn} We say that a connection $\N$ is \emph{cylindrical} if there is a
constant vector field $V$ such that
\begin{align*}
V \hook F_{\N} \equiv 0.
\end{align*}
\end{defn}

\begin{defn} \label{Xidef} Given a one-parameter family of connections $\N_s$,
$s \in I$, let
\begin{equation*}
\Xi : \mathbb{R}^{n} \times \mathbb{R}_{\geq 0} \times I : (x,t,s) \mapsto
\mathcal{F}_{x,t}(\N_s).
\end{equation*}
Where $\Xi(x,t) := \Xi(x,t,0)$.
\end{defn}

\begin{prop}\label{CMLemma}  Suppose that $\nab \in \mathfrak{S}$ is a
connection with polynomial energy growth which is not cylindrical. Given
$\epsilon > 0$ there exists $\delta > 0$ such that
\begin{equation}\label{def:Xi}
\sup \left\{ \mathcal{F}_{x_0,t_0}(\nab) : |x_0| + | \log t_0 | > \epsilon
\right\} < \la(\N) - \delta.
\end{equation}

\begin{proof} 
We show that if $\N$ is not cylindrical then $\Xi$ has a strict (global) maximum
at $(0,1)$.  We do this by showing that $(0,1)$ is the unique critical point and
then showing the second derivative at $(0,1)$ is strictly negative.  First we
show that $\Xi$ has a strict local maximum at $(0,1)$, then we show that $\Xi$
decreases along a family of paths through the space-time domain emanating from
$(0,1)$ whose union is the entire domain.

For the first step, since $\N$ is a soliton then by Proposition
\ref{zYMfirstvar} the gradient of $\Xi$ vanishes at the point $(0,1)$, which is
therefore a critical point. The second variation formula for
$\mathcal{F}_{x_0,t_0}$ computed in Proposition \ref{zhang2ndvar} applied to a
fixed $\N$ and evaluated along a path $(sy,1+sh)$ for $s>0$ and $h \in
\mathbb{R}$ yields that
\begin{equation}\label{eq:CMlem2ndvar}
\frac{\del^2}{\del s^2} \lbr \Xi (sy, 1+sh) \rbr = - 2 (1+sh) \left( 2h^2
\int_{\bRn}{|D^*F|^2 G_s dV} +  \int_{\bRn}{|y \hook F|^2 G_s dV} \right).
\end{equation}
Note that \eqref{eq:CMlem2ndvar} is nonpositive provided $(1+sh) \geq 0$. The
first term vanishes only if $h = 0$ or when $\N$ is a Yang-Mills connection, and
therefore flat by Proposition \ref{prop:YMsoliton4flat}, but we assume $\N$ is
nonflat.  Meanwhile, the second term vanishes only when $y \hook F =0$, which is
not allowed since we assume $\N$ is not cylindrical.  We thus conclude that
$\Xi$ has a strict local maximum at $(0,1)$.

We next show that for a given $y \in \mathbb{R}^{n}$ and $a \in \mathbb{R}$,
one has $\tfrac{\del}{\del s} \lbr \Xi(sy,1+ as^2) \rbr \leq 0$ for all $s > 0$
with $1 + as^2 > 0$.  We begin with Corollary \ref{5ptlemSolId}, replacing $x_0
\mapsto x_s$ and $t_0 \mapsto t_s$ and $G_0 \mapsto G_s$. We differentiate $\Xi$
using the variation formula \eqref{eq:zYMfirstvar} for $\mathcal{F}_{x_0,t_0}$
with $\dot{\gG} = 0$ to obtain
\begin{align}
\begin{split}\label{eq:Xi1}
\frac{\del}{\del s} \lbr \Xi (x_s,t_s) \rbr &=\dot{t}_s \int_{\bRn}{\left( t_s
\left(\frac{4- n}{2}\right) + \frac{| x - x_s |^2}{4}  \right) |F_{\N}|^2 G_s
dV}  +  t_s \int_{\bRn}{ \frac{\langle y , x - x_s \rangle}{2} |F_{\N}|^2 G_s
dV}.
\end{split}
\end{align}
We insert the two identities (a) and (b) of Corollary \ref{5ptlemSolId} into
\eqref{eq:Xi1}:
\begin{align*}
\frac{\del}{\del s} \lbr \Xi (x_s,t_s) \rbr
&= \dot{t}_s \int_{\bRn}{\left(t_s \left( \frac{4 -n}{2} \right) + \frac{\left|
x - x_s \right|^2}{4} + \frac{\langle y,x-x_s \rangle}{2} \right) G_s \left| F
\right|^2 dV} \\
& = \dot{t}_s \int_{\bRn} F_{pu \ga}^{\gb} F_{iu \gb}^{\ga} \left( x\left(t_s -
1 \right) + x_s \right)^p (x-x_s)^i G_s dV \\
& \hsp +  2  t_s \int_{\bRn}{F_{pu \ga}^{\gb} F_{iu \gb}^{\ga} \left(x \left(
t_s - 1 \right) + x_s \right)^p y^i G_s dV}\\
&=  \int_{\bRn}{ F_{pu \ga}^{\gb} F_{iu \gb}^{\ga} \left( x\left(t_s - 1 \right)
+ x_s \right)^p \left(  \dot{t}_s (x-x_s)^i + 2t_s y^i \right) G_s dV}.
\end{align*}
We then evaluate $\Xi$ at $x_s = sy$ and $t_s = 1 + a s^2$ to obtain
\begin{align*}
\frac{\del}{\del s} \lbr \Xi (sy,1+as^2) \rbr& = \int_{\bRn}{ F_{pu \ga}^{\gb}
F_{iu \gb}^{\ga} \left( as^2 x + sy \right)^p \left(  2as (x-sy)^i + 2(1+as^2)
y^i \right) G_s dV} \\
&=-2s \int_{\bRn} |(asx + y) \hook F|^2 G_s dV.
\end{align*}
Thus the derivative of $\Xi$ is nonpositive over the union of all paths
parametrized by $(sy,1+ys^2)$. Since these paths union to the entire space-time
domain, we conclude the result.
\end{proof}
\end{prop}

\begin{lemma} \label{Fdecaylemma} Let $f \in C^{\infty}(\bRn)$ be some function
pointwise bounded above by a polynomial $p$. Then for all $x_0 \in \bRn$,
\begin{equation*}
\lim_{t_0 \to 0} \int_{\mathbb{R}^n}{f G_0 dV} = f(x_0).
\end{equation*}
In particular we have that for a connection $\N$ on $\bRn$ with polynomial
curvature growth we have that for all $x_0 \in \bRn$,
\begin{equation*}
\lim_{t_0 \to 0} \mathcal{F}_{x_0,t_0}(\N) = 0.
\end{equation*}

\begin{proof} The first line is a well-known property of the heat kernel.  Since
$\brs{F_{\N}}^2$ has polynomial growth, we have
\begin{align*}
\lim_{t_0 \to 0} \FF_{x_0,t_0}(\N) =&\ \lim_{t_0 \to 0} t_0^2 \int_{\mathbb R^n}
\brs{F_{\N}}^2 G_0 = \left( \lim_{t_0 \to 0} t_0^2 \right) \left( \lim_{t_0 \to
0} \int_{\mathbb R^n} \brs{F_{\N}}^2 G_0 \right) = 0
\end{align*}
\end{proof}
\end{lemma}

\begin{proof}[Proof of Theorem \ref{entequiv}]
If $\N$ is not $\mathcal{F}$-stable then there is a variation $\N_s$ for $s \in
[- 2 \epsilon, 2 \epsilon]$ where $\N_0 = \N$ which satisfies the following
properties:
\begin{enumerate}
\item[(V1)]\label{V1} For each variation $\N_s$ of $\N$, the support of $\N_s -
\N$ is compact.
\item[(V2)]\label{V2} For any paths $(x_s,t_s)$ with $x_0 = 0$ and $t_0 =1$ we
have
\begin{equation}\label{eq:CMLemma0.15}
\left. \frac{\del^2}{\del s^2} \lbr \mathcal{F}_{x_s,t_s}(\N_s) \rbr
\right|_{s=0} < 0.
\end{equation}
\end{enumerate}

For the family $\N_s$, let $\Xi$ be as in Definition \ref{Xidef}.  Also, set
\begin{equation}
B^{\circ}(r):= \{ (x,t,s) : 0 < |x| + \brs{\log t} + s < r \}.
\end{equation}
With this definition we claim that there exists $\epsilon' > 0$ so that for $s
\neq 0$ and $|s| \leq \epsilon'$ one has
\begin{equation} \label{CMLemma0.15main}
\la (\N_s) := \sup_{x_0,t_0} \Xi(x_0,t_0,s) < \Xi(0,1,0) = \la(\N).
\end{equation}
Following \cite{CM} we proceed in five steps:%
\begin{enumerate}
\item[(1)] $\Xi$ has a strict local maximum at $(0,1,0)$.
\item[(2)] $\Xi(\cdot,\cdot,0)$ has a strict global maximum at $(0,1,0)$.
\item[(3)] $\frac{\del}{\del s} \lbr \Xi(x_0,t_0,s)\rbr$ is uniformly bounded on
compact sets.
\item[(4)] For $|x_0|$ sufficiently large, $\Xi(x_0,t_0,s) < \Xi(0,1,0)$.
\item[(5)] For $| \log t_0 |$ sufficiently large, $\Xi(x_0,t_0,s) < \Xi(0,1,0)$.
\end{enumerate}
Together these five pieces will yield the result as detailed at the end of the
proof.

\noindent \textbf{Proof of (1):} 
Since $\N$ is a soliton, by Corollary \ref{Fcrit}, given a path $(x_s,t_s)$ with
$(x_0,t_0) = (0,1)$ and a variation $\N_s$ of $\N$, we have
$\left. \frac{\del}{\del s}\lbr \Xi(x_s,t_s, s) \rbr \right|_{s=0} = 0$, which
implies that $(0,1,0)$ is a critical point of $\Xi$. Consider one such path of
the form $(sy,1+as)$ for $y \in \mathbb R^n$, $a \in \mathbb{R}$ and some
variation of $\N$ given by $\N_{bs}$ for some $b \neq 0$. Then we have that, by
property (V2),
\begin{equation}
\left. \frac{\del^2}{\del s^2} \lbr \Xi(sy,1+as,bs)  \rbr \right|_{s=0}=b^2
\left. \frac{\del^2}{\del s^2} \lbr \mathcal{F}_{x_s,y_s}(\N_{s})  \rbr
\right|_{s=0} \leq 0,
\end{equation}
where here $x_s = s \frac{y}{b}$ and $t_s = 1 + \frac{a}{b}s$.

Now we consider the second variation when $b = 0$. As an immediate application
of Proposition \ref{CMLemma} we have that $\left. \frac{\del^2}{\del s^2} \lbr
\Xi(sy,1+as,0) \rbr \right|_{s=0} < 0$. Therefore $\N^2 \Xi$, the Hessian of
$\Xi$, is negative definite at $(0,1,0)$, and thus $\Xi$ attains a strict local
maximum at this point. We may choose $\epsilon' \in (0,\epsilon)$ such that for
$(x_0,t_0,s) \in B^{\circ}({\epsilon'})$ we have that
\begin{equation*}
\Xi(x_0,t_0,s) < \Xi(0,1).
\end{equation*}

\noindent \textbf{Proof of (2):} This is an immediate result of Proposition
\ref{CMLemma}. Therefore, $\la(\N) = \Xi(0,1)$ and we may choose $\delta>0$ so
that for all points of the form $(x_0,t_0,0)$ outside $B^{\circ}({\epsilon' /
4})$
we have that
\begin{equation*}
\Xi(x_0,t_0) < \Xi(0,1) - \delta.
\end{equation*}

\noindent \textbf{Proof of (3):} Using Proposition \ref{zYMfirstvar}, we see
that
\begin{align*}
\frac{\del}{\del s} \lbr \Xi(x_0,t_0,s) \rbr
&=
 4 t_0^2 \int_{\bRn}{ \left\langle \dot{\gG}_s , D_s^* F_s + \left(
\frac{(x-x_0)}{2 t_0} \hook F_s \right) \right\rangle G_0 dV}.
\end{align*}
Observe that $\tfrac{\del \Xi}{\del s}$ is continuous in all three variables
$x_0$, $t_0$ and $s$ and thus uniformly bounded on compact sets.

\noindent \textbf{Proof of (4):}  By hypothesis, we may choose $R > 0$ so that
the support of $\N - \N_s$ is contained in $B(R) \subset \bRn$. Let $\rho > 0$
and consider $|x_0| > \rho + R$. Then we have that
\begin{align*}
\Xi(x_0,t_0,s) 
&= t_0^2 \int_{\mathbb{R}^n} \left| F_{\N_s} \right|^2 G_{0} dV \\
& = t_0^2 \int_{B(R)} \left| F_{\N_s} \right|^2 G_{0} dV +  t_0^2
\int_{\mathbb{R}^n \backslash B(R)} \left| F_{\N} \right|^2 G_{0} dV \\
& \leq t_0^{\frac{4-n}{2}} (4 \pi)^{- \frac{n}{2}} \int_{B(R)} \left| F_{\N_s}
\right|^2 e^{-\frac{|x-x_0|^2}{4 t_0}} dV +  \Xi(x_0,t_0,0) \\
&\leq t_0^{\frac{4-n}{2}} (4 \pi)^{- \frac{n}{2}} e^{-\frac{\rho^2}{4
t_0}}\int_{B(R)} \left| F_{\N_s} \right|^2  dV +  \Xi(x_0,t_0,0).
\end{align*}
By compactness of the domain $B(R) \times [-2 \epsilon , 2 \epsilon ]$ we know
that $\int_{B(R)} \left| F_{\N_s} \right|^2  dV < C_R$ for some $C_R \in
\mathbb{R}$. Therefore we conclude that
\begin{equation}\label{eq:sc4Xiest}
\Xi(x_0,t_0,s) \leq (4 \pi)^{-\frac{n}{2}} C_R t_0^{\frac{4-n}{2}}
e^{-\frac{\rho^2}{4 t_0}} + \Xi(x_0,t_0,0).
\end{equation}
Define the quantity
\begin{equation}\label{eq:sc4muRdef}
\mu_{\rho}(\tau) := \tau^{\frac{4-n}{2}} e^{- \frac{\rho^2}{4 \tau}}.
\end{equation}
We note in particular that
\begin{equation*}
\mu_1 \left( \frac{\tau}{\rho^2} \right) =\left( \frac{\tau}{\rho^2}
\right)^{\frac{4-n}{2}} e^{- \frac{\rho^2}{4 \tau}} = \rho^{n-4}
\mu_{\rho}(\tau).
\end{equation*}
The function $\mu_1$ is clearly continuous and therefore bounded and also
satisfies the following limit for $\ga \in \{ \infty, 0 \}$,
\begin{align*}
\lim_{\tau \to \ga} \mu_1 \left( \tau \right) = \lim_{\tau \to \ga}
\tau^{\frac{4-n}{2}} e^{-\frac{1}{4 \tau }} = 0.
\end{align*}
We thus conclude that
\begin{equation}
\lim_{\rho \to \infty} \left( \sup_{\tau> 0} \mu_{\rho}(\tau) \right) =
\lim_{\rho \to \infty} \sup_{\tau> 0}\left(   \rho^{4-n} \mu_1 \left(
\frac{\tau}{\rho^2} \right) \right) = 0.
\end{equation}
Therefore, as a consequence of (2) combined with this above limit, we conclude
that for $|x_0|$ sufficiently large we have that $\Xi(x_0,t_0,s) < \Xi(0,1,0)$,
as desired.

\noindent \textbf{Proof of (5):} We first perform the following manipulation
\begin{gather}
\begin{split}\label{eq:s5Ximanipulation}
\Xi(x_0,t_0,s) 
&= t_0^2 \int_{\mathbb{R}^n} \left| F_{\N_s} \right|^2 G_{0} dV \\
& = t_0^2 \int_{B(R)} \left| F_{\N_s} \right|^2 G_{0} dV +  t_0^2
\int_{\mathbb{R}^n \backslash B(R)} \left| F_{\N_s} \right|^2 G_{0} dV \\
& \leq t_0^{\frac{4-n}{2}} (4 \pi)^{- \frac{n}{2}} \int_{B(R)} \left| F_{\N_s}
\right|^2 G_0 dV +  \Xi(x_0,t_0,0) \\
& \leq C_R t_0^{\frac{4-n}{2}} (4 \pi)^{- \frac{n}{2}} +  \Xi(x_0,t_0,0).
\end{split}
\end{gather}
As a result of this, we also obtain the estimate
\begin{equation*}
\sup_{t_0 \geq 1} \Xi(x_0,t_0,s) \leq C_R (4 \pi)^{-\frac{n}{2}} + \la (\N).
\end{equation*}
We break into two cases.  First, suppose $t_0$ is very large.   Combining
\eqref{eq:s5Ximanipulation} with part (2) we obtain the claim.  The case when
$t_0$ is small, in particular $t_0 \leq 1$, is more difficult. Appealing to
Lemma \ref{zYMfirstvar} with $\dot{t_0} = 1$, we have that for some fixed $R>0$
\begin{align*}
\frac{\del}{\del t_0} \lbr \Xi(x_0, t_0, s) \rbr 
&=   \int_{\mathbb{R}^n \backslash B(R)} \left( t_0 \left( \frac{4-n}{2} \right)
+ \frac{|x-x_0|^2}{4} \right) |F_{\N}|^2 G_0 dV \\
& \hsp +  \int_{B(R)} \left( t_0 \left( \frac{4-n}{2} \right) +
\frac{|x-x_0|^2}{4} \right) |F_{\N_s}|^2 G_0 dV  \\
& \geq  t_0 \int_{\mathbb{R}^n } \left( \frac{4-n}{2} \right) |F_{\N_s}|^2 G_0 +
\left( \frac{4-n}{2} \right) C_R t_0.
\end{align*}
Arguing similarly to Lemma \ref{Fdecaylemma}, the integral on the left is
bounded. Furthermore since $t_0 \leq 1$ we have that for some $C_0 \in
\mathbb{R}$,
\begin{equation}\label{eq:Xit0deriv}
\frac{\del }{\del t_0} \lbr \Xi(x_0,t_0,s) \rbr \geq - C_0.
\end{equation}
Note that $C_0$ is independent of $x_0$, $t_0$, and $s$ subject to the
restriction $|x_0|<R$.  Using Lemma \ref{Fdecaylemma} and Step (2), we have that
\begin{equation}
\Xi(0,1,0) = \la(\N) > 0.
\end{equation}
Choose $\ga > 0$ so that $3 \ga < \la(\N)$, and choose $t_{\ga} =
\frac{\ga}{C_0}$.
For any $x \in \bRn$ and $s \in [-\epsilon, \epsilon]$, by Lemma
\ref{Fdecaylemma} there exists some $t_{x,s} > 0$ such that for all $t_0 \leq
t_{x,s}$ we have $|\Xi(x,t_0,s)|<\ga$.

On the set $\overline{B(R+1)} \times [-\epsilon, \epsilon]$ we will construct a
finite open cover as follows. The cover consists of balls $b_i$ of radius
$r_i>0$ centered at $(x_i,t_i)$. Each $b_i$ has an associated time $ t_i  \leq
\min \{ t_{\ga}, 1 \}$ where
\begin{enumerate}
\item Given $(x,s)$ there exists and index $i(x,s)$ such that $(x,s) \in
b_{i(x,s)}$.
\item For each $b_i$ the associated $t_i$ is such that
\begin{equation*}
\left. \Xi(x,t_i,s) \right|_{b_i} < \ga.
\end{equation*}
Note that this choice follows from the existence of $t_{x,s}$ and the continuity
of $\Xi$.
\end{enumerate}
Choosing a finite subcover of the $b_i$'s we let $\bar{t}$ be the minimum of all
corresponding $t_i$. Then as a result of the derivative \eqref{eq:Xit0deriv} we
have that for any triple $(x,t_0,s)$ with $s \in [-\epsilon,\epsilon]$, and $x
\in \overline{B_{R}}$, and $t_0 \leq \bar{t}$,
\begin{align*}
\Xi(x,t_0,s)   \leq \Xi(x, t_{i(x,s)},s) + C_0 \left(t_{i(x,s)} - t_0 \right) 
\leq 2\ga  < \la(\N).
\end{align*}
Claim (5) follows.\\

Given claims (1)-(5) we finish the proof by dividing the domain into regions
corresponding to the size of $|x_0| + \brs{\log t_0}$.  Using (1), when $s$ is
sufficiently small there exists some $r > 0$ such that $\Xi(x_0,t_0,s) <
\Xi(x_0,t_0,0)$ for $(x_0,t_0)$ within the following region
\begin{equation*}
\mathfrak{R}_1 := \{ (x_0,t_0) : |x_0| + |\log t_0| < r \}.
\end{equation*}
Using (4) and (5) there exists an $R>0$ such that $\Xi(x_0,t_0,s) <
\Xi(x_0,t_0,0)$ for $(x_0,t_0)$ in the following region.
\begin{equation*}
\mathfrak{R}_2 := \{ (x_0,t_0) : |x_0| + |\log t_0| > R \}.
\end{equation*}
Therefore it remains to consider 
\begin{equation}
\mathfrak{R}_3 := \{ (x_0,t_0) : R > |x_0| + |\log t_0| > r \}.
\end{equation}
Given $(x_0,t_0) \in \mathfrak{R}_3$, we know by (2) that $\Xi(x_0,t_0,0) <
\la(\N)$, and by (3) that the $s$ derivative of $\Xi$ is uniformly bounded. So
we may choose a $\delta > 0$ such that $\Xi$ restricted to the region
$\mathfrak{R}_3 \times [-\delta,\delta]$ is bounded above by $\la(\N)$. 
Therefore, \eqref{CMLemma0.15main} holds on $\bigcup_{i=1}^3\mathfrak{R}_i$ and
as this union constitutes the entire space-time domain, the result follows.
\end{proof}

\section{Gastel shrinkers} \label{gastelsec}

In this section we recall Gastel's construction \cite{Gastel} of
$SO(n)$-shrinking solitons, and compute their entropies.  Let $\{ \zeta_i \}_{i
=1}^n \subset \mathfrak{so}(n)$, be the basis given by, for $\ga, \gb \in [1,n]
\cap \mathbb{N}$,
\begin{equation*}
\zeta_{i \ga}^{\gb} := \delta_i^{\gb} x_{\ga} - \delta_{i \ga} x^{\gb}.
\end{equation*}
Now let $r$ denote the radius on $\bRn$, and fix some function $\eta :
[0,\infty) \to \mathbb{R}$. Consider the $SO(n)$-equivariant connections $\nab$
with coefficient matrices given by
\begin{equation} \label{gastelshr}
\gG_{i\ga}^{\gb}(x) := - \frac{\eta(r)}{r^2} \zeta_{i\ga}^{\gb}(x).
\end{equation}

\begin{prop} For $5 \leq n \leq 9$, and
\begin{align} \label{gasteleta}
\eta(r) := \frac{r^2}{a_n r^2 + b_n},
\end{align}
where
\begin{equation}\label{eq:andef}
a_n := \sqrt{\frac{n-2}{8}}, \quad b_n = 3(n-2) -
\tfrac{1}{\sqrt{2}}(n+2)(n-2)^{1/2} \geq 0,
\end{equation}
the connection $\N$ defined by \ref{gastelshr} is a shrinking soliton.
\begin{proof}
As computed in \cite{Gastel} \S 2.1, under the ansatz of (\ref{gastelshr}), the
Yang-Mills flow reduces to
\begin{equation}\label{eq:YMeta}
\eta_t = \eta_{rr} + (n-3) \frac{\eta_r}{r} - (n-2) \frac{\eta (\eta - 1)(\eta
-2)}{r^2}.
\end{equation}
With $\eta(r)$ as in (\ref{gasteleta}) and $a_n, b_n$ as in (\ref{eq:andef}), we
set
\begin{align*}
\eta(r,t) = \eta \left( \frac{r}{\sqrt{-t}} \right).
\end{align*}
We now compute various derivatives.
\begin{align}
\begin{split}\label{eq:etaderiv}
\eta(r,t) &=  \frac{r^2}{a_n r^2 - b_n t}\\
\eta_t(r,t) &= \frac{b_n r^2t}{(a_n r^2 - b_n t)^{2}}\\
\eta_r(r,t) 
&= \frac{-2rb_nt}{(a_n r^2 - b_nt)^2},\\
\eta_{rr}(r,t) &= \frac{(6a_n b_n r^2 t + 2 b_n^2 t^2)}{(a_n r^2 - b_n t)^3}\\
(\eta - 1) &= \frac{(1-a_n)r^2 + b_n t}{(a_n r^2 - b_n t)},\\
(\eta - 2) &= \frac{(1-2a_n)r^2 + 2b_n t}{(a_n r^2 - b_n t)}.
\end{split}
\end{align}
We plug the identities of \eqref{eq:etaderiv} into \eqref{eq:YMeta} and obtain
\begin{align*}
0 &=\frac{1}{(a_n r^2 - b_n t)^3} \left(6 a_nb_n r^2 t + 2 b_n^2 t^2 +
(n-3)\left(-2a_n b_n t r^2 + 2 b_n^2 t^2 \right)\right)\\
& \hsp + \frac{1}{(a_n r^2 - b_n t)^3} \left(-
(n-2)\left(\left(1-a_n\right)\left(1-2a_n \right)r^4 + \left(3 - 4a_n\right) b_n
t r^2 + 2 b_n^2 t^2 \right)-a_n b_n r^4 + b_n^2 t r^2 \right).
\end{align*}
We collect up the coefficients within the numerator of $r$ to various powers.
The coefficient of $1$ is given by
\begin{align*}
2 b_n^2 t^2& + (n-3)2 b_n^2 t^2 - (n-2) 2 b_n^2 t^2 \\
&= (1 + n - 3 - n + 2) 2 b^2 t^2 \\
& = 0.
\end{align*}
The coefficient of $r^2$ is given by
\begin{align*}
6a_nb_nt& + (n-3)\left(-2a_nb_nt \right)- (n-2)\left(3 -4a_n \right)b_nt + b_n^2
t\\
&= 2a_nb_nt(n+2) - 3(n-2)b_nt + b_n^2 t.
\end{align*}
We will compute and then combine portions of the above quantity. First, consider
the product
\begin{align*}
a_n b_n &= \left( \frac{n-2}{8} \right)^{1/2} \left( 3(n-2) +
\frac{1}{\sqrt{2}}\left(n+2 \right)(n-2)^{1/2} \right) \\
&= \frac{3}{\sqrt{2}}(n-2)^{3/2} + \frac{1}{4}(n-2)(n+2).
\end{align*}
Therefore we have
\begin{equation}\label{eq: YMetaverQ1}
2a_nb_n (n+2) = \frac{6}{\sqrt{2}}(n+2)(n-2)^{3/2} - \frac{1}{2}(n-2)(n+2)^2.
\end{equation}
Additionally there is 
\begin{align}
\begin{split}\label{eq: YMetaverQ2}
b_n^2 t&= \left( 3(n-2) - \tfrac{1}{\sqrt{2}}(n+2)(n-2)^{1/2}\right)^2 t \\
&= 9 (n-2)^2 t - \frac{6}{\sqrt{2}} (n+2)(n-2)^{3/2}t + \frac{1}{2} (n+2)^2(n-2)
t.
\end{split}
\end{align}
Lastly we have
\begin{equation}\label{eq: YMetaverQ3}
-3(n-2)b_n t= - 9 (n-2)^2 t + \frac{1}{2} (n+2)^2 (n-2)t.
\end{equation}
Combining together we have
\begin{equation}
\eqref{eq: YMetaverQ1} + \eqref{eq: YMetaverQ2} + \eqref{eq: YMetaverQ3} = 0,
\end{equation}
as desired. Lastly we consider the coefficient of $r^4$.
\begin{align*}
-a_n b_n& - (n-2)(1-a_n)(1-2 a_n)\\
&= - \frac{3}{2 \sqrt{2}}(n-2)^{3/2} + \frac{1}{4}(n-2)(n+2) - (n-2) +
\frac{3}{2 \sqrt{2}}(n-2)^{3/2} - \frac{(n-2)}{4} \\
& =\frac{(n-2)}{4} \left( (n+2) -4 - (n-2)\right)\\
&=0.
\end{align*}
Thus we conclude that $\eta(r,t)$ satisfies \eqref{eq:YMeta} and thus this
particular connection is a solution to Yang-Mills flow. Next we verify that
$\nabla_{-1}(x)$ is a soliton by verifying the scaling law
\eqref{solitonscallaw} of Lemma \ref{wnkovescallawsoliton}. Observe that
\begin{align*}
\la \gG_{i \ga}^{\gb} (\la x, \la^2t ) &=- \la \frac{\eta( \la r, \la^2 t
)}{\la^2 r^2} \zeta_{i \ga}^{\gb}(\la x) \\
&=- \la \frac{1}{\la^2 r^2} \left( \frac{\la^2 r^2}{a_n \la^2 r^2 - \la^2 b_n
t}\right) \left( \delta_i^{\gb} \la x_{\ga} -\delta_{i \ga} \la x^{\gb}\right)
\\
&= -\frac{\eta(r,t)}{r^2} \zeta_{i \ga}^{\gb}(x) \\
&= \gG_{i \ga}^{\gb}(x,t).
\end{align*}
We conclude that $\nab_{-1}$ is a soliton.
\end{proof}
\end{prop}

\begin{prop} Let $\N_n$ denote the Gastel soliton on $\mathbb R^n$.  Then the
entropy of $\N_n$ is approximately
\begin{align*}
 \begin{tabular}{|c|c|}
  \hline
  $n$ & $\gl(\N_n)$\\
  \hline
  $5$ & $638.121$\\
  \hline
  $6$ & $716.109$\\
  \hline
  $7$ & $929.899$\\
  \hline
  $8$ & $1292.44$\\
  \hline
  $9$ & $1865.98$\\
  \hline
 \end{tabular}
\end{align*}

\begin{proof} We note that by Proposition \ref{CMLemma} it suffices to compute
$\FF_{0,1}(\N)$, which we do numerically.  In the midst of the computation to
obtain \eqref{eq:YMeta} within \cite{Gastel} \S 2.1, if we set $\phi = \frac{-
\eta}{r^2}$, then
\begin{equation}\label{eq:phiF}
F_{jk \ga}^{\gb} = \left( 2 \phi + r^2 \phi^2 \right)\left( \delta_{k}^{\gb}
\delta_{\ga j}  - \delta_{k \ga} \delta^{\gb}_{j} \right) + \left(
\frac{\phi_r}{r} - \phi^2 \right) \left( \delta_k^{\gb} x_{\ga}x_j + \delta_{j
\ga} x^{\gb}x_k - \delta_{k \ga} x^{\gb}x_j - \delta_j^{\gb} x_{\ga}x_k  
\right).
\end{equation}
We compute $\mathcal{F}_{x_0,t_0}$ by first considering
\begin{align*}
& |F|^2 =-g^{ik}g^{j \ell} \left( F_{ij \ga}^{\gb} \right) \left( F_{k \ell
\gb}^{\ga}  \right)\\
& =-g^{ik}g^{j \ell} \left( \left( 2 \phi + r^2 \phi^2 \right)\left(
\delta_{j}^{\gb} \delta_{\ga i}  - \delta_{j \ga} \delta^{\gb}_{i} \right) +
\left( \frac{\phi_r}{r} - \phi^2 \right) \left( \delta_j^{\gb} x_{\ga}x_i   +
\delta_{i \ga} x^{\gb}x_j - \delta_{j \ga} x^{\gb}x_i - \delta_i^{\gb}
x_{\ga}x_j \right) \right)\\
& \hsp \times \left(\left( 2 \phi + r^2 \phi^2 \right)\left( \delta_{\ell}^{\ga}
\delta_{\gb k}  - \delta_{\ell \gb} \delta^{\ga}_{k} \right) + \left(
\frac{\phi_r}{r} - \phi^2 \right) \left( \delta_{\ell}^{\ga} x_{\gb}x_k   +
\delta_{k \gb} x^{\ga}x_{\ell} - \delta_{\ell \gb} x^{\ga}x_{k} - \delta_k^{\ga}
x_{\gb}x_{\ell} \right)\right) \\
&= - g^{ik}g^{j \ell} \left( 2 \phi + r^2 \phi^2 \right)^2 \left[ \left(
\delta_{j}^{\gb} \delta_{\ga i}  - \delta_{j \ga} \delta^{\gb}_{i} \right)\left(
\delta_{\ell}^{\ga} \delta_{\gb k}  - \delta_{\ell \gb} \delta^{\ga}_{k} \right)
\right]_{T_1} \\
& \hsp - g^{ik}g^{j \ell} \left( 2 \phi + r^2 \phi^2 \right)\left(
\frac{\phi_r}{r} - \phi^2 \right) \left[ \left( \delta_{j}^{\gb} \delta_{\ga i} 
- \delta_{j \ga} \delta^{\gb}_{i} \right) \left( \delta_{\ell}^{\ga} x_{\gb}x_k 
 + \delta_{k \gb} x^{\ga}x_{\ell} - \delta_{\ell \gb} x^{\ga}x_{k} -
\delta_k^{\ga} x_{\gb}x_{\ell} \right) \right]_{T_2} \\
& \hsp - g^{ik}g^{j \ell}\left( \frac{\phi_r}{r} - \phi^2 \right) \left( 2 \phi
+ r^2 \phi^2 \right) \left[ \left( \delta_j^{\gb} x_{\ga}x_i   + \delta_{i \ga}
x^{\gb}x_j - \delta_{j \ga} x^{\gb}x_i - \delta_i^{\gb} x_{\ga}x_j \right)\left(
\delta_{\ell}^{\ga} \delta_{\gb k}  - \delta_{\ell \gb} \delta^{\ga}_{k} \right)
\right]_{T_3} \\
 & \hsp - g^{ik}g^{j \ell}  \left( \frac{\phi_r}{r} - \phi^2 \right)^2 \left[
\left( \delta_{\ell}^{\ga} x_{\gb}x_k   + \delta_{k \gb} x^{\ga}x_{\ell} -
\delta_{\ell \gb} x^{\ga}x_{k} - \delta_k^{\ga} x_{\gb}x_{\ell} \right) \left(
\delta_j^{\gb} x_{\ga}x_i   + \delta_{i \ga} x^{\gb}x_j - \delta_{j \ga}
x^{\gb}x_i - \delta_i^{\gb} x_{\ga}x_j \right) \right]_{T_4}.
\end{align*}
We first expand $T_1$.
\begin{align*}
T_1 &= \left( \delta_{j}^{\gb} \delta_{\ga i}\delta_{\ell}^{\ga} \delta_{\gb k}
- \delta_{j}^{\gb} \delta_{\ga i}\delta_{\ell \gb} \delta^{\ga}_{k} - \delta_{j
\ga} \delta^{\gb}_{i} \delta_{\ell}^{\ga} \delta_{\gb k} + \delta_{j \ga}
\delta^{\gb}_{i} \delta_{\ell \gb} \delta^{\ga}_{k} \right) \\
&= \left( \delta_{j k} \delta_{\ell i} - \delta_{j \ell} \delta_{ik} - \delta_{j
\ell} \delta_{ik}+ \delta_{jk} \delta_{i \ell} \right) \\
&= 2 \left( \delta_{j k} \delta_{\ell i} - \delta_{j \ell} \delta_{ik} \right).
\end{align*}
Contracting via multiplication by $g^{ik}g^{j \ell}$ yields
\begin{align*}
 g^{ik}g^{j \ell} T_1 &= 2 \left( \delta_{j k}^2 - \delta_{j j} \delta_{kk}
\right)\\
 &=- 2 n\left( n - 1 \right).
\end{align*}
Next we expand $T_2$.
\begin{align*}
T_2 &= \delta_{j}^{\gb} \delta_{\ga i} \left( \delta_{\ell}^{\ga} x_{\gb}x_k   +
\delta_{k \gb} x^{\ga}x_{\ell} - \delta_{\ell \gb} x^{\ga}x_{k} - \delta_k^{\ga}
x_{\gb}x_{\ell} \right)\\
& \hsp - \delta_{j \ga} \delta^{\gb}_{i} \left( \delta_{\ell}^{\ga} x_{\gb}x_k  
+ \delta_{k \gb} x^{\ga}x_{\ell} - \delta_{\ell \gb} x^{\ga}x_{k} -
\delta_k^{\ga} x_{\gb}x_{\ell} \right)\\
&=  \delta_{i \ell} x_{j}x_k   + \delta_{j k} x_i x_{\ell} - \delta_{j \ell} x_i
x_{k} - \delta_{ik}x_{j}x_{\ell}  - \delta_{j \ell}x_{i}x_k   - \delta_{ik}
x_{j}x_{\ell} + \delta_{i \ell} x_{j}x_{k} + \delta_{jk} x_{i} x_{\ell} \\
&= 2 \left( \delta_{i \ell} x_{j}x_k   + \delta_{j k} x_i x_{\ell} - \delta_{j
\ell} x_i x_{k} - \delta_{ik}x_{j}x_{\ell} \right).
\end{align*}
Contracting via multiplication by $g^{ik}g^{j \ell}$ yields
\begin{align*}
 g^{ik}g^{j \ell} T_2 &= 2 \left( |x|^2   + |x|^2 - n |x|^2 - n |x|^2 \right) \\
 &=- 4 \left( n- 1 \right) |x|^2.
\end{align*}
Expanding $T_3$ we obtain
\begin{align*}
T_3 &=  \left( \delta_j^{\gb} x_{\ga}x_i   + \delta_{i \ga} x^{\gb}x_j -
\delta_{j \ga} x^{\gb}x_i - \delta_i^{\gb} x_{\ga}x_j \right)
\delta_{\ell}^{\ga} \delta_{\gb k}  \\
& \hsp - \left( \delta_j^{\gb} x_{\ga}x_i   + \delta_{i \ga} x^{\gb}x_j -
\delta_{j \ga} x^{\gb}x_i - \delta_i^{\gb} x_{\ga}x_j \right) \delta_{\ell \gb}
\delta^{\ga}_{k}\\
&=  \left( \delta_j^{\gb} \delta_{\ell}^{\ga} \delta_{\gb k}  x_{\ga}x_i   +
\delta_{i \ga} \delta_{\ell}^{\ga} \delta_{\gb k}  x^{\gb}x_j - \delta_{j \ga}
\delta_{\ell}^{\ga} \delta_{\gb k}  x^{\gb}x_i -
\delta_i^{\gb}\delta_{\ell}^{\ga} \delta_{\gb k}  x_{\ga}x_j \right) \\
& \hsp - \left( \delta_j^{\gb} \delta_{\ell \gb} \delta^{\ga}_{k}x_{\ga}x_i   +
\delta_{i \ga} \delta_{\ell \gb} \delta^{\ga}_{k} x^{\gb}x_j - \delta_{j
\ga}\delta_{\ell \gb} \delta^{\ga}_{k} x^{\gb}x_i - \delta_i^{\gb} \delta_{\ell
\gb} \delta^{\ga}_{k}x_{\ga}x_j \right) \\
&= 2 \left( \delta_{j k}  x_{\ell}x_i   + \delta_{i \ell} x_{k} x_j - \delta_{j
\ell} x_{k} x_i - \delta_{ik} x_{\ell}x_j \right).
\end{align*}
Contracting by multiplying $g^{ik} g^{j \ell}$ we obtain
\begin{equation*}
 g^{ik}g^{j \ell} T_3 = - 4 \left( n -1 \right)|x|^2.
\end{equation*}
Lastly we expand $T_4$ and obtain
\begin{align*}
T_4 &= \delta_{\ell}^{\ga} x_{\gb}x_k  \left( \delta_j^{\gb} x_{\ga}x_i   +
\delta_{i \ga} x^{\gb}x_j - \delta_{j \ga} x^{\gb}x_i - \delta_i^{\gb}
x_{\ga}x_j \right) \\
& \hsp + \delta_{k \gb} x^{\ga}x_{\ell} \left( \delta_j^{\gb} x_{\ga}x_i   +
\delta_{i \ga} x^{\gb}x_j - \delta_{j \ga} x^{\gb}x_i - \delta_i^{\gb}
x_{\ga}x_j \right) \\
& \hsp   - \delta_{\ell \gb} x^{\ga}x_{k} \left( \delta_j^{\gb} x_{\ga}x_i   +
\delta_{i \ga} x^{\gb}x_j - \delta_{j \ga} x^{\gb}x_i - \delta_i^{\gb}
x_{\ga}x_j \right) \\
& \hsp  - \delta_k^{\ga} x_{\gb}x_{\ell} \left( \delta_j^{\gb} x_{\ga}x_i   +
\delta_{i \ga} x^{\gb}x_j - \delta_{j \ga} x^{\gb}x_i - \delta_i^{\gb}
x_{\ga}x_j \right)\\
&=  \left( x_{j}x_k  x_{\ell}x_i   + \delta_{i \ell} x_k x_j |x|^2 - \delta_{j
\ell} x_k x_i |x|^2 - x_{i}x_k  x_{\ell}x_j \right) \\
& \hsp + \left( \delta_{j k} x_{\ell} x_i |x|^2  + x_{i}x_{\ell} x_k x_j -  x_j
x_{\ell} x_k x_i - \delta_{ik} x_{\ell} x_j |x|^2 \right) \\
& \hsp   -\left( \delta_{j \ell} x_{k} x_i |x|^2   +  x_{i}x_{k} x_{\ell}x_j - 
x_j x_{k} x_{\ell} x_i - \delta_{i \ell} x_{k} x_j |x|^2 \right) \\
& \hsp  -  \left(   x_{j}x_{\ell} x_{k}x_i   + \delta_{i k} x_{\ell} x_j |x|^2 -
\delta_{j k}  x_{\ell} x_i |x|^2 - x_{i}x_{\ell}x_{k}x_j \right)\\
&=2 \left(  \delta_{i \ell} x_k x_j |x|^2 - \delta_{j \ell} x_k x_i |x|^2 +
\delta_{j k} x_{\ell} x_i |x|^2  - \delta_{ik} x_{\ell} x_j |x|^2 \right).
\end{align*}
Contracting indices $g^{ik} g^{j \ell}$ we obtain
\begin{equation*}
 g^{ik}g^{j \ell} T_4 = -4 \left( n-1 \right)|x|^4.
\end{equation*}
Therefore, we conclude that 
\begin{align}
\begin{split}\label{eq:normFshrinkex}
| F |^2 &=  2 n\left( n - 1 \right) \left( 2 \phi + r^2 \phi^2 \right)^2 +  8
\left( n -1 \right) r^2 \left( 2 \phi + r^2 \phi^2 \right)\left(
\frac{\phi_r}{r} - \phi^2 \right)  + 4 \left( n-1 \right) r^4 \left(
\frac{\phi_r}{r} - \phi^2 \right)^2.
\end{split}
\end{align}
We substitute in $\phi(r) = \frac{-\eta(r)}{r^2}$. Recall that $\phi_r = \del_r
\lbr \frac{\eta}{r^2}\rbr = \frac{-4(n-1)}{r^4} \left( r \eta_r - 2 \eta -
\eta^2 \right)$. Therefore we conclude that
\begin{align*}
|F|^2 & = \frac{2n(n-1)}{r^4} \left( 2 \eta + \eta^2 \right)^2 +
\frac{8(n-1)\left( 2 \eta + \eta^2 \right)}{r^4} \left( r \eta_r - 2 \eta -
\eta^2 \right) + \frac{4(n-1)}{r^4} \left( r \eta_r - 2 \eta - \eta^2 \right)^2
\\
&= \frac{2(n-1)}{r^4}\left( n(4 \eta^2 + 4 \eta^3 + \eta^4) + 4(2r \eta \eta_r +
r \eta^2 \eta_r - 4 \eta^2 - 4 \eta^3 - \eta^4) \right)\\
& \hsp + \frac{2(n-1)}{r^4} \left( 2 (r^2 \eta_r^2 - 4r \eta \eta_r -2 r \eta^2
\eta_r  + 4 \eta^2 + 4 \eta^3 + \eta^4) \right) \\
&= \frac{2(n-1)}{r^4} \left( \eta^2 \left( 4n -16 + 8 \right) + \eta^3 \left( 4n
-16+ 8 \right) + \eta^4 (n - 4 + 2) +(8 - 8)  r \eta \eta_r +2 r^2 \eta_r^2 +
(4-4) r \eta^2 \eta_r \right)\\
&= \frac{2(n-1)}{r^4} \left( 4\eta^2 \left( n -2 \right) + 4\eta^3 \left( n -2
\right) + \eta^4 (n - 2)  + 2 r^2 \eta_r^2 \right).
\end{align*}
Using the definition of $\eta$ from (\ref{gasteleta}) and incorporating the
corresponding identities computed in \eqref{eq:etaderiv} we have
\begin{align*}
|F|^2 &= \frac{2(n-1)}{r^4(a_n r^2 + b_n)^4} \left( 4 r^4 (a_n r^2 + b_n)^2
\left( n -2 \right) + 4 r^6 (a_n r^2 + b_n) \left( n -2 \right) + r^8 (n - 2)  +
8 r^4 b_n^2 \right)\\
&=  \frac{2(n-1)(n-2)}{(a_n r^2 + b_n)^4} \left( 4 (a_n^2 r^4 + a_n b_n r^2 + 
b_n^2) + 4 (a_n r^4 +  b_n r^2) + r^8 + \frac{8 r^4 b_n^2}{(n-2)}\right)\\
&=  \frac{8(n-1)(n-2)}{(a_n r^2 + b_n)^4}  \left[b_n^2 + r^2 \left( a_n b_n +
b_n \right)+ r^4 \left(a_n^2 + a_n + \frac{2 b_n^2}{(n-2)} \right) + \frac{1}{4}
r^8\right].
\end{align*}
Using this in the definition of $\FF_{0,1}$ and integrating numerically yields
the result.
\end{proof}
\end{prop}

\bibliographystyle{hamsplain}

\end{document}